\RequirePackage{ifpdf}
\ifpdf
\documentclass[preprint,12pt]{elsarticle}

\usepackage{amssymb}
\usepackage[T1]{fontenc}
\usepackage{fourier}
\usepackage[english]{babel} 
\usepackage[protrusion=true,expansion=true]{microtype} 
\usepackage{amsmath,amsfonts,amsthm} %
\usepackage{wrapfig}
\usepackage{lscape}
\usepackage{rotating}
\usepackage{float}
\usepackage{graphicx}
\usepackage{hyperref}
\usepackage{tikz}
\usetikzlibrary{fit,calc,positioning,shapes,arrows}

\usepackage{verbatim}

\DeclareFontFamily{U}{MnSymbolC}{} \DeclareSymbolFont{MnSyC}{U}{MnSymbolC}{m}{n} \DeclareFontShape{U}{MnSymbolC}{m}{n}{ <-6> MnSymbolC5 <6-7> MnSymbolC6 <7-8> MnSymbolC7
<8-9> MnSymbolC8 <9-10> MnSymbolC9 <10-12> MnSymbolC10 <12-> MnSymbolC12}{} \DeclareMathSymbol{\intprod}{\mathbin}{MnSyC}{'270}
\newtheorem{thm}{Theorem}
\newtheorem{propo}{Proposition}
\newtheorem{defi}{Definition}

\newtheorem{lem}[thm]{Lemma}
\newdefinition{rmk}{Remark}
\newproof{pf}{Proof}
\newproof{pot}{Proof of Theorem \ref{thm2}}
\newtheorem{example}{Example}[section]
\def\l{{\lambda}}

\def\A{{\mathbf{A}}}
\def\campoA{{\mathbf{A}}}
\def\Yu{{\mathbf{X}_1}}
\def\Yd{{\mathbf{X}_2}}
\def\Zu{{\mathbf{Y}_1}}
\def\Zd{{\mathbf{Y}_2}}
\def\Xu{{\mathbf{X}_1}}
\def\Xd{{\mathbf{X_2}}}

\begin{document}

\begin{frontmatter}




\title{$\lambda-$Sy\-mme\-tries and integrability by quadratures}


\author[conchi]{C. Muriel}
\ead[conchi]{concepcion.muriel@uca.es}
\author[conchi]{J. L. Romero}
\ead[conchi]{juanluis.romero@uca.es}
\author[conchi]{A. Ruíz}
\ead[conchi]{adriruse@gmail.com}
\address[conchi]{Departamento de Matem\'aticas, Universidad de C\'adiz,
11510 Puerto Real, Spain.}

\begin{abstract}
It is investigated how two  (standard or generalized) $\lambda-$sy\-mme\-tries of a given
second-order ordinary differential equation  can be used to solve the equation by quadratures. The method is based on the construction of two commuting generalized symmetries for this equation by using both $\lambda-$sy\-mme\-tries.   The functions used in that construction  are related with integrating factors of the reduced and
auxiliary equations associated to the $\lambda-$sy\-mme\-tries. These functions can also be used to derive a Jacobi
last multiplier and two integrating factors for the given equation.

Some examples illustrate the  method; one of them is included in  the XXVII case of the Painlev\'e-Gambier classification. An explicit expression of its general solution in terms of two fundamental sets of solutions for two related second-order linear equations is also obtained.
\end{abstract}

\end{frontmatter}



\section{Introduction}\label{introduction}

A remarkable application of the Lie group theory to differential equations is that the general solution of an $n$th-order ordinary differential equation (ODE) that admits an $n$th-dimensional solvable symmetry algebra can be found by
means of $n$ successive quadratures \cite{olver1993applications,ibragimovlibro,stephani}. However, the existence of non-trivial point symmetries is  not a necessary condition for the
integrability by quadratures of ODEs. An example of a family of second-order ODEs integrable by quadratures whose point symmetry group is trivial was firstly provided in
\cite{artemioecu}. Such example, and many others that appeared later in the literature \cite{artemiosis, govinder97, nucci2008}, motivated the need of
developing a more general context than the framework of Lie point symmetries to explain the integrability by quadratures of a given ODE.

The integrability of many of these equations can be explained by using  $\lambda-$sy\-mme\-tries (also called $\mathcal{C}^{\infty}-$symmetries
\cite{muriel01ima1, muriel03lie}). This concept was introduced in \cite{muriel01ima1} and it is based on a prolongation formula that generalizes the usual prolongation of vector fields. $\mathcal{C}^{\infty}-$symmetries can be used to reduce the order of ODEs as Lie point symmetries do and 
 have been widely studied and generalized from very different points of view (see \cite{gaetatwisted,1751-8121-42-3-035210,doi:10.1142/S1402925109000303,0305-4470-37-27-007} and the references therein). 
 
 However, it seems that  no studies have been made on the consequences of having two or more $\mathcal{C}^\infty-$symmetries for a given ODE. In fact, the reduction process associated to one of the
$\lambda-$symmetries is independent of the corresponding one for any of the remaining $\mathcal{C}^\infty-$symmetries.

The integrability of $2$nd-order ODEs admitting two $\mathcal{C}^{\infty}-$symmetries is studied in this paper. Section \ref{s1} includes a review of the basic notions (limited to second-order ODEs) and the extension of the initial notion of $\mathcal{C}^{\infty}-$symmetry introduced in \cite{muriel01ima1} to consider $\lambda-$prolongations of generalized vector fields.
In Sections \ref{s3} and \ref{s4}  a systematic procedure to construct two commuting generalized symmetries from two given (standard or generalized) $\mathcal{C}^{\infty}-$sy\-mme\-tries of a second-order ODE is provided. According to \cite{muriel2009}, two independent first integrals of the ODE can be found, not necessarily by quadrature, by means of a
procedure that uses each $\mathcal{C}^{\infty}-$sy\-mme\-try separately. In Section \ref{s5}  the two $\mathcal{C}^{\infty}-$sym\-me\-tries   and the
functions that appear in the construction of the generalized symmetries are used simultaneously to find such first integrals by qua\-dra\-ture.

Several significant objects  in the analytical study of the ODEs arise as an immediate consequence  of the described procedure. For instance, 
 a well-known result on the relationships between Jacobi last multipliers
\cite{whittaker1988treatise, nucci2005jacobi, MR2606129, NucciLeachCGroup1, NucciLeachCGroup2} and Lie point symmetries is that the knowledge of two
independent symmetries provides an explicit formula for a Jacobi last multiplier. In Section \ref{secJacobi}, a new explicit formula for a Jacobi last multiplier of $2$nd-order ODEs that
involve two $\mathcal{C}^{\infty}-$sy\-mme\-tries admitted by the equation 
is   provided.

In Section \ref{secJacobi} the known  $\lambda-$symmetries and the functions that appear in the construction of the generalized symmetries are used to obtain two integrating factors for the second-order ODE. The cited functions are closely related to  integrating factors of the reduced and auxiliary
equations that appear in the reduction processes associated to the given generalized $\mathcal{C}^{\infty}-$sy\-mme\-tries.

The results in this paper are applied to a subclass of equations in the XXVII case of the Painlev\'e-Gambier classification \cite{ince}. In the general case the equations in this subclass do
not admit Lie point symmetries, but any equation of the family admits two non-equivalent $\mathcal{C}^{\infty}-$sy\-mme\-tries, which have been recently found by the
authors of \cite{guha2013lambda}.  Such $\mathcal{C}^{\infty}-$sy\-mme\-tries are used to illustrate the procedure of integration of the equations by quadratures. The study of that family of equations has been carried out through several examples to illustrate the different steps that appear in the method.  As a consequence, the general solution of any of the equations in the family can be expressed in terms of two fundamental sets of solutions for two second-order linear equations. 
Explicit expressions for a Jacobi last multiplier and the integrating factors of the reduced and auxiliary equations are
also provided.

In Section \ref{sect7} a step-by-step description of the procedure to facilitate its practical application is presented. Several examples show how the method can be used to solve equations lacking Lie point symmetries or which admit just one Lie point symmetry.

\section{Preliminaries}\label{s1}

Throughout this paper $M$ will denote an open subset of the space of the independent and dependent variables $(x,u)$ of a given $2$nd-order ODE:
\begin{equation}
\label{orden2} u_{xx}=\phi(x,u,u_x),
\end{equation} where the subscript denotes derivation with respect to $x.$

The vector field on the jet space $M^{(1)}$ associated to equation (\ref{orden2}) is defined by  $\mathbf{A}=\partial_x+u_x\partial_u+\phi(x,u,u_x)\partial_{u_{x}}$. 
  The total derivative operator with respect to $x$ is  defined by
$$\mathbf{D}_x=\partial_x+u_x\partial_u+u_{xx}\partial_{u_{x}}+\cdots.$$

Recall \cite{muriel01ima1} that if  $\mathbf{v}=\xi(x,u)\partial_x+\eta(x,u)\partial_u$ is a vector field on $M$ and $\lambda=\lambda(x,u,u_x)$ is a smooth function defined on $M^{(1)}$ then the 
first-order $\lambda-$pro\-lon\-ga\-tion of $\mathbf{v}$ is the vector field $\mathbf{v}^{[\lambda,(1)]}$ on $M^{(1)}$ defined by 
\begin{equation}
\label{glambdaprol} 
\mathbf{v}^{[\lambda,(1)]}=\mathbf{v}+\bigl((\mathbf{D}_x+\lambda)(\eta)-(\mathbf{D}_x+\lambda)(\xi)u_x \bigr)\partial_{u_{x}}.
\end{equation}
The functions $\xi$ and $\eta$ will be called the infinitesimals of $\mathbf{v}$.
Observe that for $\lambda=0,$ (\ref{glambdaprol}) is the standard first-order prolongation $\mathbf{v}^{(1)}$ of $\mathbf{v}$ \cite{stephani}.

The 
pair 
$(\mathbf{v},\lambda)$ is a $\mathcal{C}^{\infty}-$sy\-mme\-try (also called a $\lambda-$symmetry) of equation (\ref{orden2})
if
\begin{equation}\label{corcheteA}
[\mathbf{v}^{[\lambda,(1)]},\mathbf{A}]=\lambda\mathbf{v}^{[\lambda,(1)]}+\mu \mathbf{A},
\end{equation}
 where $\mu=-(\mathbf{A}+\lambda)(\xi)$. 
 For $\lambda=0$, if $\mathbf{v}$ satisfies the condition (\ref{corcheteA}) then $\mathbf{v}$ is a standard Lie point symmetry of (\ref{orden2}) \cite{stephani}.

%

When a $\mathcal{C}^{\infty}-$sy\-mme\-try $(\mathbf{v},\lambda)$ is known, a method to solve the given ODE proceeds as follows:  %
if $y=y(x,u), w=w(x,u,u_x)$ are two invariants of $\mathbf{v}^{[\lambda,(1)]},$ then (\ref{orden2}) can be written as a reduced equation $\Delta(y,w,w_y)=0.$ Then the general solution of (\ref{orden2}) arises from the general solution $w={H}(y,C)$ of the
reduced equation by solving the corresponding auxiliary equation $w(x,u,u_x)=H\bigl(y(x,u),C\bigr),$ where $C\in \mathbb{R}$ (see Theorem 3.2 in \cite{muriel01ima1} for
details). This method generalizes the classical Lie method and has been successfully applied to integrate or reduce the order of many ODEs lacking Lie point symmetries \cite{muriel01ima1, muriel03lie}.

If  $\xi=\xi(x,u,u_x)$ and  $\eta=\eta(x,u,u_x)$ are two smooth functions defined on $ M^{(1)}$ then 
$$\mathbf{v}=\xi(x,u,u_x)\partial_x+\eta(x,u,u_x)\partial_u$$
is a well-defined vector field on $ M^{(1)}$. If $\lambda=\lambda(x,u,u_x)$ is also a smooth function on $ M^{(1)}$ then by considering the vector field $\mathbf{A}$, the vector field $\mathbf{v}$ can be prolonged by using a formula  similar to (\ref{glambdaprol}) but changing $\mathbf{D}_x$ by $\mathbf{A}$: i.e. 
\begin{equation}
\label{glambdaprol1} 
\mathbf{v}^{[\lambda,(1)]}=\mathbf{v}+\bigl((\mathbf{A}+\lambda)(\eta)-(\mathbf{A}+\lambda)(\xi)u_x \bigr)\partial_{u_{x}},
\end{equation}
which is a well-defined vector field on  $M^{(1)}$.  The pair  $(\mathbf{v},\lambda)$ 
 will be called a \textit{generalized} $\mathcal{C}^{\infty}-$sym\-me\-try (or a generalized $\lambda-$symmetry) of equation (\ref{orden2}) if 
 (\ref{corcheteA}) holds, where $\mathbf{v}^{[\lambda,(1)]}$ is defined by (\ref{glambdaprol1}).  It is clear that if the pair $(\mathbf{v},\lambda)$ is a generalized $\mathcal{C}^\infty-$symmetry of equation (\ref{orden2}) for the function $\lambda=0$ then the vector field
 $\mathbf{v}^{[0,(1)]}=\mathbf{v}^{(1)}$ is a generalized symmetry of equation  (1) [3].


According to (\ref{corcheteA}), any given generalized $\mathcal{C}^{\infty}-$sy\-mme\-try $(\mathbf{v},\lambda)$ of (\ref{orden2}) defines a vector field
$\mathbf{v}^{[\lambda,(1)]}$ such that the system $\left\{\mathbf{A},\mathbf{v}^{[\lambda,(1)]}\right\}$ is in involution. By Fröbenius Theorem \cite{olver1993applications}, the system
$\left\{\mathbf{A},\mathbf{v}^{[\lambda,(1)]}\right\}$ is integrable and the integral submanifold is locally defined by a  first integral,
$I\in \mathcal{C}^{\infty}(M^{(1)}),$ such that 

$$\mathbf{A}(I)=\mathbf{v}^{[\lambda,(1)]}(I)=0.$$
In this case, $I$ will be called a first integral of $\mathbf{A}$ associated to
$(\mathbf{v},\lambda).$


A given first integral of $\campoA$ can be associated to several
generalized $\mathcal{C}^{\infty}-$sy\-mme\-tries. This fact is related to the notion of equivalence of $\mathcal{C}^{\infty}-$sy\-mme\-tries established in
\cite{muriel2009,muriel09wascom}. Next  this concept is extended for generalized $\mathcal{C}^{\infty}-$sy\-mme\-tries:

\begin{defi}\label{defeq}{\rm
 Two generalized $\mathcal{C}^{\infty}-$sy\-mme\-tries $(\mathbf{v}_1,\lambda_1)$ and $(\mathbf{v}_2,\lambda_2)$ of the equation (\ref{orden2}) will be called 
$\mathbf{A}-$equivalent if the set of vector fields  $\left\{\campoA,{\mathbf{v}_{1}^{[\lambda_1,(1)]}},{\mathbf{v}_{2}^{[\lambda_2,(1)]}}\right\}$ is dependent 
over $\mathcal{C}^{\infty}(M^{(1)}).$ In this case the notation  $(\mathbf{v}_1,\lambda_1) \stackrel{\campoA}{\sim}(\mathbf{v}_2,\lambda_2)$  will be used.}
\end{defi}

It follows from Definition \ref{defeq} that the first integrals of $\campoA$ associated to equi\-va\-lent generalized $\mathcal{C}^{\infty}-$sy\-mme\-tries are functionally dependent. Consequently,  a first integral of $\campoA$ associated to a generalized $\mathcal{C}^{\infty}-$sy\-mme\-try $(\mathbf{v},\lambda)$ can be found by using any element of its class of equivalence. A particularly simple element
in the equivalence class of a given generalized $\mathcal{C}^{\infty}-$sy\-mme\-try  $(\mathbf{v},\lambda)$ is the
 pair $\bigl(\partial_u,\lambda+\mathbf{A}(Q)/Q\bigr),$ where $Q=\eta-\xi u_x$ denotes the characteristic of $\mathbf{v}=\xi\partial_x+\eta\partial_u.$ This pair is called the
\textit{canonical} representative of the equivalence class  of $(\mathbf{v},\lambda)$ \cite{muriel2009}.

Since (\ref{orden2}) is a $2$nd-order equation, its general solution can be obtained by considering two first integrals $I_1$ and $I_2$ of $\mathbf{A}$ which are respectively associated to two known non-equivalent generalized  $\mathcal{C}^{\infty}-$sy\-mme\-tries $(\bar{\mathbf{v}}_1,\bar{\lambda}_1)$ and $(\bar{\mathbf{v}}_2,\bar{\lambda}_2)$ of (\ref{orden2}).
In the following sections  a procedure to compute such first integrals by quadratures will be described.

\section{Commuting generalized $\mathcal{C}^{\infty}-$sy\-mme\-tries}\label{s3}

Assume that equation (\ref{orden2}) admits two non-equivalent generalized  $\mathcal{C}^{\infty}-$sy\-mme\-tries $(\bar{\mathbf{v}}_1,\bar{\lambda}_1)$ and $(\bar{\mathbf{v}}_2,\bar{\lambda}_2)$ and   let $\mathcal{A}_1$ and $\mathcal{A}_2$ be  their respective  $\mathbf{A}$-equi\-va\-len\-ce classes. By using their canonical representatives it can be considered, without loss of generality,  that such  $\mathcal{C}^{\infty}-$symmetries are of the form  $(\partial_u,\lambda_1)\in \mathcal{A}_1$ and $(\partial_u,\lambda_2)\in \mathcal{A}_2$, where $\lambda_1\neq \lambda_2$ because   $\mathcal{A}_1\neq \mathcal{A}_2$.   In this case the functions $\lambda_1$ and $\lambda_2$ are particular solutions
of the determining equation (see equation (7) in \cite{muriel2009}):
\begin{equation}
\label{eqdet2} \lambda_x+\lambda_u u_x+\lambda_{u_x} \phi+\lambda^{2}=\phi_{u}+\lambda\phi_{u_x}.
\end{equation}
   In order to simplify the notation, for $i=1,2$,  $(\partial_u)^{[\lambda_{i},(1)]}$ will be denoted by $\mathbf{X}_i$:

\begin{equation}\label{Ys}
\mathbf{X}_i=(\partial_u)^{[\lambda_{i},(1)]}=\partial_u+\lambda_{i}\partial_{u_x}, \quad (i=1,2). \end{equation}
As a direct consequence of former definitions, it can be checked that 
\begin{equation}\label{corchetesY}
[\mathbf{X}_1,\mathbf{A}]=\lambda_1\mathbf{X}_1, \quad [\mathbf{X}_2,\mathbf{A}]=\lambda_2\mathbf{X}_2, \quad
[\mathbf{X}_1,\mathbf{X}_2]=\rho(\mathbf{X}_1-\mathbf{X}_2),
\end{equation}
where the function $\rho=\rho(x,u,u_x)$ is given by
\begin{equation}\label{ro} 
\rho=\frac{\mathbf{X}_1(\lambda_2)-\mathbf{X}_2(\lambda_1)}{\lambda_1-\lambda_2}.
\end{equation} 

Since it is assumed that $\lambda_1\neq \lambda_2$, then $\mathbf{X}_1-\mathbf{X}_2\neq0$ and therefore the vector fields 
$\mathbf{X}_1$ and $\mathbf{X}_2$ commute if and only if $\rho=0$. In this case, by applying the Jacobi identity
to the vectors fields $\bigl\{\mathbf{A},\mathbf{X}_1,\mathbf{X}_2\bigr\}$ and by using  (\ref{corchetesY}),
\begin{equation}
\label{jacobi}
\begin{array}{lll}
0&=&\bigl[\Yu,[\Yd,\A]\bigr]+\bigl[\Yd,[\A,\Yu]\bigr]+\bigl[\A,[\Yu,\Yd]\bigr]\\
& =& [\Yu,{\l_2}\,\Yd] -[\Yd,{\l_1}\,\Yu]+[\A,0]\\
&=& \Yu({\lambda_2})\cdot\Yd-\Yd({\lambda_1})\cdot \Yu.
\end{array}
\end{equation} Since the vector fields $\Yu,\Yd$ are not proportional, because $\lambda_1\neq \lambda_2$, (\ref{jacobi}) implies that
\begin{equation}
\label{lambdasinvariantes} \Yu({\lambda_2})=\Yd({\lambda_1})=0.
\end{equation}
In other words, the vector fields $\mathbf{X_1}$ and $\mathbf{X}_2$ commute if and only if (\ref{lambdasinvariantes}) holds.

%
%
%
The following lemma will be used to prove the existence of two $\mathcal{C}^{\infty}-$sy\-mme\-tries  which are $\A-$equivalent to $(\partial_u,\lambda_1)$ and
$(\partial_u,\lambda_2)$  respectively and are such that the corresponding first-order $\lambda-$pro\-lon\-ga\-tions commute:

\begin{lem}\label{lema1} {\rm With the above notations, if $f_1=f_1(x,u,u_x)$ and $f_2=f_2(x,u,u_x)$ are two functions such that \begin{equation}\label{f1f2}
\frac{\mathbf{X}_1(f_2)}{f_2}= \frac{\mathbf{X}_2(f_1)}{f_1}=\rho,
\end{equation} where $\rho$ is given by (\ref{ro}), then 
\begin{equation}
\label{yes}
[f_1\mathbf{X}_1,\mathbf{A}]=\rho_1(f_1\mathbf{X}_1),\qquad [f_2\mathbf{X}_2,\mathbf{A}]=\rho_2(f_2\mathbf{X}_2),\qquad [f_1\mathbf{X}_1,f_2\mathbf{X}_2]=0,
\end{equation}
where
\begin{equation} \label{lambdasnew}\rho_i=\lambda_i-\frac{\mathbf{A}(f_i)}{f_i}, \quad\mbox{
for } i=1,2.\end{equation}
}
\end{lem}
\begin{proof}
By the properties of the Lie bracket, for any functions $f_1$ and $f_2$ the relations $$\begin{array}{lll}
[f_1\Yu,f_2\Yd]&=&f_1\,f_2\,[\Yu,\Yd]+f_1\Yu(f_2)\cdot\Yd-f_2\Yd(f_1)\cdot \Yu\\ &=&f_1\,f_2\,\rho(\Yu-\Yd)+f_1\Yu(f_2)\cdot\Yd-f_2\Yd(f_1)\cdot \Yu\\
&=&f_2\left(\rho-\frac{\Yd(f_1)}{f_1}\right)\cdot f_1\Yu-f_1\left(\rho-\frac{\Yu(f_2)}{f_2}\right)\cdot f_2\Yd
\end{array}$$ hold. If $f_1,f_2$ satisfy (\ref{f1f2}), then $[f_1\mathbf{X}_1,f_2\mathbf{X}_2]=0.$

The first two relations in (\ref{yes}) can be proved similarly.
\end{proof}

In order to simplify the notations, if $f_1$ and $f_2$ satisfy  (\ref{f1f2}) then  the vector fields $f_1\mathbf{X}_1$ and  $f_2\mathbf{X}_2$ will be denoted by $\mathbf{Y}_1$ and $\mathbf{Y}_2$ respectively:
 $\mathbf{Y}_i= f_i\mathbf{X}_i,$ for $i=1,2.$  
 With these notations, equations in (\ref{yes}) can be written as 
\begin{equation}\label{corchetesnuevos}
[\mathbf{Y}_1,\mathbf{A}]=\rho_1\mathbf{Y}_1,\quad [\mathbf{Y}_2,\mathbf{A}]=\rho_2\mathbf{Y}_2,\quad [\mathbf{Y}_1,\mathbf{Y}_2]=0,
\end{equation}
where  $\rho_1$ and $\rho_2$ are given by (\ref{lambdasnew}).

By (\ref{corcheteA}), the first two  relations in (\ref{corchetesnuevos}) show that \begin{equation}
\label{Y} \mathbf{Y}_i=f_i\mathbf{X}_i=(f_i\partial_u)^{[\rho_i,(1)]} \quad \mbox{for} \quad i=1,2.
\end{equation} 

According to Definition \ref{defeq}, the pairs $(f_1\partial_u,\rho_1)$ and  $(f_2\partial_u,\rho_2)$ are two generalized $\mathcal{C}^{\infty}-$symmetries of (\ref{orden2}) which are  $\mathbf{A}-$equi\-valent to $(\partial_u,\lambda_1)$ and $(\partial_u,\lambda_2)$ respectively.

%

Previous discussion shows that if $(\partial_u,\lambda_1)\in \mathcal{A}_1$ and $(\partial_u,\lambda_2)\in \mathcal{A}_2$ are two non-equivalent   $\mathcal{C}^{\infty}-$symmetries of (\ref{orden2}) then $(f_1\partial_u,\rho_1)\in \mathcal{A}_1$ and  $(f_2\partial_u,\rho_2)\in \mathcal{A}_2$   and by (\ref{Y}) their first-order $\lambda-$prolongations commute.


For further reference, the next theorem collects the main aspects of former discussion.

\begin{thm}\label{teocommuting}
{\rm Let   $(\partial_u,\lambda_1)$ and $(\partial_u,\lambda_2)$ be the canonical representatives of  two non-equivalent generalized  $\mathcal{C}^{\infty}-$sym\-me\-tries of (\ref{orden2}) and let $\mathcal{A}_1$ and $\mathcal{A}_2$
be their respective equivalence classes. Denote $\mathbf{X}_i=(\partial_u)^{[\lambda_i,(1)]},$ for $i=1,2.$ Let $\rho$ be the function defined by (\ref{ro}) and let $f_1,f_2$ be two functions satisfying (\ref{f1f2}).  Then:
\begin{enumerate}
\item $(f_1\partial_u,\rho_1)\in \mathcal{A}_1$  and $(f_2\partial_u,\rho_2)\in \mathcal{A}_2$, where $\rho_1,\rho_2$ are given by (\ref{lambdasnew}). 
\item By denoting $\mathbf{Y}_1=(f_1\partial_u)^{[\rho_1,(1)]}$ and $\mathbf{Y}_2=(f_2\partial_u)^{[\rho_2,(1)]}$, the relations (\ref{corchetesnuevos}) hold. In particular, $\mathbf{Y}_1,\mathbf{Y}_2$ commute.

\end{enumerate}}
\end{thm}

 The above-described method is illustrated in the next example by  constructing  two  commuting
generalized $\mathcal{C}^{\infty}-$sym\-me\-tries from two known $\mathcal{C}^{\infty}-$sym\-me\-tries of an equation.

\begin{example}\label{eje1}
{\rm Let us consider the family of $2$nd-order equations
\begin{equation}\label{painleve}
u_{xx}-\frac{1}{2\,u}u_{x}^2+ 2\,u u_{x} +\frac{1}{2}u^{3}-F(x)  u+\frac{1}{2\,u}=0,
\end{equation} which is a particular case of the XXVII equation in the Painlevé-Gambier classification \cite{ince}.

For an arbitrary  function $F=F(x)$, equation (\ref{painleve}) does not admit Lie point symmetries. Nevertheless,  it has been shown in \cite{guha2013lambda} that 
two non-equivalent $\mathcal{C}^{\infty}-$sy\-mme\-tries of (\ref{painleve}), $(\partial_u,\lambda_1)$ and
$(\partial_u,\lambda_2),$ are defined by: 
\begin{equation}\label{lambdasymeje}
\lambda_1={\frac {u_x}{u_{{}}}}-u_{{}}+\frac{1}{u},\quad\lambda_2={\frac {u_x}{u_{{}}}}-u_{{}}-\frac{1}{u}.
\end{equation} 
It can be checked that these two functions are particular solutions of the co\-rres\-ponding determining equation (\ref{eqdet2}).
For this case, the vector fields (\ref{Ys}) satisfy:
$$[\Yu,\Yd]=\frac{2}{u}\left(\Yu-\Yd\right).$$
Since $\rho=2/u$  does not depend on $u_{x},$ 
 two functions $f_1$ and $f_2$ satisfying (\ref{f1f2}) can be easily computed:
\begin{equation}\label{f1f2ej}
f_1(x,u,u_{x})=f_2(x,u,u_{x})=u^2.
\end{equation}
The vector fields $\Zu=u^2\Xu$ and $\Zd=u^2\Xd$ become
\begin{equation}\label{Yejemplo}
\begin{array}{l}
\Zu=u^2\partial_u+(u_x u-u^3+u)\partial_{u_x
},\\
\Zd=u^2\partial_u+(u_x u-u^3-u)\partial_{u_x
}\end{array}\end{equation} 
and satisfy $[\Zu,\Zd]=0.$ 
According to (\ref{Y}), it  can be written  $\Zu=(u^2\partial_u)^{[\rho_1,(1)]}$ and $\Zd=(u^2\partial_u)^{[\rho_2,(1)]}$ for the functions
$$\rho_1=\lambda_1-\frac{\campoA(u^2)}{u^2}=-\frac {u_x}{u_{{}}}-u_{{}}+\frac{1}{u} \quad\mbox{and}\quad \rho_2=\lambda_2-\frac{\campoA(u^2)}{u^2}=-\frac {u_x}{u_{{}}}-u_{{}}-\frac{1}{u}.$$ 

 Consequently, 
 two new  $\mathcal{C}^{\infty}-$symmetries    of equation (\ref{painleve}) have been constructed, $(u^2\partial_u,\rho_1)$ and $(u^2\partial_u,\rho_2),$ which are  $\campoA-$equivalent to $(\partial_u,\lambda_1)$ and $(\partial_u,\lambda_2)$ respectively and satisfy  $$[\mathbf{Y}_1,\mathbf{Y}_2]= \left[(u^2\partial_u)^{[\rho_1,(1)]},(u^2\partial_u)^{[\rho_2,(1)]}\right]=0.$$
}
\hfill\qedsymbol \end{example}

An important advantage of the pairs $(f_1\partial_u,\rho_1)\in \mathcal{A}_1$ and $(f_2\partial_u,\rho_2)\in \mathcal{A}_2$ constructed in Theorem \ref{teocommuting} is that
 $\Zu=(u^2\partial_u)^{[\rho_1,(1)]}$ and $\Zd=(u^2\partial_u)^{[\rho_2,(1)]}$ can be simultaneously straightened by using quadratures alone.
The next objective is to  search for two independent functions $w_1=w_1(x,u,u_x)$ and $w_2=w_2(x,u,u_x)$ such that in the local system of coordinates  $\{x,w_1,w_2\}$ of $M^{(1)}$
the vector fields $\mathbf{Y}_1$ and $\mathbf{Y}_2$ can be written as
\begin{equation}
\label{nuevo1}
\mathbf{Y}_1=\partial_{w_2},\qquad \mathbf{Y}_2=\partial_{w_1}.
\end{equation}
The conditions (\ref{nuevo1}) would imply 
\begin{equation}
\label{nuevo2} \mathbf{Y}_1(w_2)=1,\qquad \mathbf{Y}_2(w_2)=0,
\end{equation}
and 
\begin{equation}
\label{nuevo3} \mathbf{Y}_1(w_1)=0,\qquad \mathbf{Y}_2(w_1)=1.
\end{equation}
Since $\mathbf{Y}_1=f_1\partial_u+f_1\lambda_1\partial_{u_x}$ and $\mathbf{Y}_2=f_2\partial_u+f_2\lambda_2\partial_{u_x}$, equations (\ref{nuevo2})
would imply:
$$\begin{array}{l}
\Zu(w_2)=({w_2})_u\Zu(u)+({w_2})_{u_x}\Zu(u_x) =f_1({w_2})_u+\lambda_1f_1({w_2})_{u_x}=1,\\{}\\
\Zd(w_2)=({w_2})_u\Zd(u)+({w_2})_{u_x}\Zd(u_x) =f_2({w_2})_u+\lambda_2f_2({w_2})_{u_x}=0.
\end{array} $$
and then $({w_2})_u$ and $({w_2})_{u_x}$ would be defined by 
\begin{equation} \label{m2}
({w_2})_u=\dfrac{\lambda_2}{f_1(\lambda_2-\lambda_1)},\qquad ({w_2})_{u_x}=-\dfrac{1}{f_1(\lambda_2-\lambda_1)}.
\end{equation}
The local existence of such function $w_2$ is warranted because the mixed partials coincide:
$$\left(\dfrac{\lambda_2}{f_1(\lambda_2-\lambda_1)}\right)_{u_x}=-\left(\dfrac{1}{f_1(\lambda_2-\lambda_1)}\right)_u;$$
this can be checked by using that $\mathbf{X}_2(f_1)=\rho_1f_1$ and a straightforward  calculation. It is clear that the conditions (\ref{m2}) imply that $w_2$ can be constructed by quadratures if $f_1$ is known.  

A similar reasoning can be followed to prove the local existence of a function $w_1$ satisfying (\ref{nuevo1}), by using  system (\ref{nuevo3}) instead of (\ref{nuevo2}); $w_1$  satisfies 
\begin{equation}\label{m1}
({w_1})_{u}=\dfrac{\lambda_1}{f_2(\lambda_1-\lambda_2)},\qquad ({w_1})_{u_{x}}=-\dfrac{1}{f_2(\lambda_1-\lambda_2)}.
\end{equation}

For further reference, the following proposition collects some   properties of the  considered coordinate system $\{x,w_1,w_2\}$.

\begin{propo}
\label{rectificacion1} 
{\rm Let   $(\partial_u,\lambda_1)$ and $(\partial_u,\lambda_2)$ be the canonical representatives of two non-equivalent generalized  $\mathcal{C}^{\infty}-$sym\-me\-tries of (\ref{orden2}).
 Consider the vector fields $\Zu$ and $\Zd$ given in Theorem \ref{teocommuting}. A local system of coordinates $\{x,w_1,w_2\}$ on $M^{(1)}$ in which
$\Zu=\partial_{w_2}$ and $\Zd=\partial_{w_1}$ can be constructed by  quadratures by using the infinitesimals of $\Zu$ and $\Zd$ .}
\end{propo}

 The functions $f_1$ and $f_2$ that satisfy (\ref{f1f2}) let the construction  by quadratures of the
invariants $w_1,w_2$ of  $\Zu,\Zd$, respectively.  Since $\Zu=f_1\Yu$ and $\Zd=f_2\Yd,$ the functions $w_1,w_2$ are also invariants of $\Yu$ and $\Yd,$
respectively.   These invariants can be used to construct reduced equations associated to each $\mathcal{C}^{\infty}-$sym\-me\-try as follows.

   The first equation in (\ref{corchetesnuevos})
shows that if $w_1$ is an invariant of $\Zu$ then $\phi_1=\A(w_1)$ is also an invariant of $\Zu$ and $\phi_1$ can be expressed in terms of $\{x,w_1\}.$ Similarly,
$\phi_2=\A(w_2)$ is an invariant of $\mathbf{Y}_2$ which can be expressed in terms of $\{x,w_2\}.$  Therefore, in terms of the local system of coordinates $\{x,w_1,w_2\}$, the vector field $\mathbf{A}$ becomes
\begin{equation}\label{Anuevo}
{\mathbf{A}}=\partial_x+{\phi}_1(x,w_1)\partial_{w_1}+{\phi}_2(x,w_2)\partial_{w_2}.
\end{equation}
In consequence
\begin{equation} \label{reducidas} (w_1)_x={\phi}_1(x,w_1) \quad
\mbox{ and } \quad (w_2)_x={\phi}_2(x,w_2)
\end{equation} are two reduced equations associated to $(\partial_u,\lambda_1)$ and $(\partial_u,\lambda_2),$ respectively. The respective vector fields associated to the reduced equations (\ref{reducidas}) are \begin{equation}
\label{camposreducidas} \campoA_1=\partial_x+{\phi}_1(x,w_1)\partial_{w_1} \quad \mbox{ and } \quad \campoA_2=\partial_x+{\phi}_2(x,w_2)\partial_{w_2}.
\end{equation}

In the next example it is  illustrated the procedure given in the proof of the Proposition \ref{rectificacion1} to compute invariants of the vector fields (\ref{Yejemplo}) by using (\ref{f1f2ej}) and quadratures alone. These invariants will be used to compute the
reduced equations associated to the  $\mathcal{C}^{\infty}-$symmetries defined by (\ref{lambdasymeje}).

\begin{example}\label{eje2}
{\rm By proceeding with the study of equation (\ref{painleve}) made in Example \ref{eje1}, let $\Zu$ and $\Zd$ be
 the vector fields given in (\ref{Yejemplo}). The functions $\lambda_1$ and $\lambda_2$ given in (\ref{lambdasymeje}) and the functions $f_1=f_2=u^2$ given in  (\ref{f1f2ej}) will be used to construct the systems corresponding to (\ref{m1}) and (\ref{m2}):
$$ \begin{array}{ll}
({w_1})_{u}=\dfrac{u_x-u^2+1}{2u^2},& ({w_1})_{u_{x}}=-\dfrac{1}{2
u},\\
 ({w_2})_u=-\dfrac{u_x-u^2-1}{2u^2},& ({w_2})_{u_x}=\dfrac{1}{2
 u}.\end{array}
 $$
 Both systems can be easily solved by quadratures and  \begin{equation}\label{wejemplo}
 w_1=-{\frac {u_x+{u_{{}}}^{2}+1}{2u_{{}}}}\quad \mbox{and} \quad w_2={\frac {u_x+{u_{{}}}^{2}-1}{2u_{{}}}}\end{equation} are, respectively, particular solutions.

It can be checked that, for this example,   the vector fields  (\ref{Yejemplo}) become $\Zu=\partial_{w_2}$ and $\Zd=\partial_{w_1}$ in variables $\{x,w_1,w_2\}.$
In these variables the vector field $\campoA$ associated to equation (\ref{painleve}) is
$${\mathbf{A}}=\partial_x+\left({w_1}^{2}-\dfrac{1}{2}\,\bigl(F(x) +1\bigr)\right)\partial_{w_1}-\left({w_2}^{2}-\dfrac{1}{2}\,\bigl(F(x)  -1\bigr)\right)\partial_{w_2},$$
   which according to (\ref{reducidas}) provides the following reduced equations

\begin{equation}\label{reduceje} (w_1)_x=\displaystyle{w_1^{2}-\dfrac{1}{2}\,\bigl(F ( x ) +1\bigr)},\qquad (w_2)_x=-w_2^{2}+\dfrac{1}{2}\,\bigl(F ( x ) -1\bigr).
\end{equation} These equations were also obtained by the authors of \cite{guha2013lambda} by following a different procedure.} \hfill\qedsymbol \end{example}

Thus far, the relations (\ref{corchetesY}) have been simplified by changing the initial generalized $\mathcal{C}^{\infty}-$sy\-mme\-tries by other equivalent
ones that define the vector fields $\Zu$ and $\Zd$ given in (\ref{Y}). As a consequence, reduced equations associated to the $\mathcal{C}^{\infty}-$sy\-mme\-tries can be constructed by quadratures.  

In the next section it is shown  how to construct two standard generalized symmetries that are equivalent to the generalized $\mathcal{C}^{\infty}-$sy\-mme\-tries that
define $\Zu$ and $\Zd.$ Once these generalized symmetries are known, the initial equation (\ref{orden2}) can be completely integrated by quadratures.

\section{Commuting generalized symmetries from generalized $\mathcal{C}^{\infty}-$sy\-mme\-tries}\label{s4}

In this section it is investigated the existence of two  non identically zero functions $g_1,g_2\in \mathcal{C}^{\infty}(M^{(1)})$ such that
\begin{equation}
\label{zetas}
[g_1\Zu, \mathbf{A}]=[g_2\Zd, \mathbf{A}]=[g_1\Zu, g_2\Zd]=0,
\end{equation}
 where $\Zu$ and $\Zd$ are the vector fields constructed in Theorem \ref{teocommuting} from two known $\mathcal{C}^{\infty}-$symmetries of equation (\ref{orden2}).
In case of existence of such functions $g_1,g_2$, equations   (\ref{corchetesnuevos})  and the properties of the Lie bracket would imply that $g_1,g_2$ satisfy

\begin{equation}
\label{sistemas}
\begin{array}{ll}
\left\{\begin{array}{l}
\mathbf{A}(g_1)=\rho_1 g_1,\\
\Zd(g_1)=0,
\end{array}\right. &
\left\{\begin{array}{l}
\mathbf{A}(g_2)=\rho_2 g_2,\\
\Zu(g_2)=0,
\end{array}\right.
\end{array}
\end{equation}
where the functions $\rho_1$ and $\rho_2$ are defined by  (\ref{lambdasnew}). Conversely, it can be checked that if $g_1,g_2\in \mathcal{C}^{\infty}(M^{(1)})$ are respectively solutions of the systems in  (\ref{sistemas}) then $g_1,g_2$ satisfy (\ref{zetas}).

Observe that each  system in  (\ref{sistemas}) is a coupled system of  two first-order partial differential equations with three independent variables whose compatibility is, \textit{a priori}, not obvious. The local  system of
coordinates $\{x,w_1,w_2\}$  obtained in Proposition
\ref{rectificacion1}  will be used to reduce simultaneously each system in (\ref{sistemas}) to a single partial differential equation with two independent variables, and the compatibility of the systems in  (\ref{sistemas}) will be straightforward.

\begin{lem} \label{lema8}{\rm There exist two functions ${g_1}={g_1}(x,u,u_x)$ and ${g_2}={g_2}(x,u,u_x)$ which satisfy the corresponding system in (\ref{sistemas}).
}
\end{lem}
\begin{proof}

 Recall that in the coordinates $\{x,w_1,w_2\}$ given by  Proposition
 \ref{rectificacion1} the vector field $\Zu$ (resp. $\Zd$) can be written as  $\Zu=\partial_{w_2}$ (resp. $\Zd=\partial_{w_1}$). Since $g_1,g_2$ must satisfy $\Zd(g_1)=\Zu(g_2)=0,$ the functions $g_1$ and $g_2$ in coordinates $\{x,w_1,w_2\}$ must be of the form
${g}_1={g}_1(x,w_2)$ and ${g}_2={g}_2(x,w_1),$ respectively. 
Since $g_1$ (resp. $g_2$) must satisfy $\mathbf{A}(g_1)=\rho_1g_1$ (resp. $\mathbf{A}(g_2)=\rho_2g_2$), the expression of $\campoA$ in coordinates $\{x,w_1,w_2\}$ given in (\ref{Anuevo}), suggests that $g_1=g_1(x,w_2)$ (resp. $g_2=g_2(x,w_1)$)  could be any particular solution of the first-order partial differential equation
\begin{equation}
\label{pdered} ({{g}_1})_x+({{g}_1})_{w_2}{\phi}_2={{g}_1}\,({{\phi}_2})_{w_2}  \quad \bigl( \mbox{resp.}\quad
({{g}_2})_x+({{g}_2})_{w_1}{\phi}_1= {{g}_2}\,({{\phi}_1})_{w_1} \bigr).
\end{equation}
It can be checked that if $\bar{g}_1(x,w_1)$ (resp. $\bar{g}_2(x,w_2)$) is a solution for (\ref{pdered}) then, 
by writing $w_1$ and $w_2$  in terms of $\{x,u,u_x\}$, the functions   $g_1=\bar{g}_1(x,w_2(x,u,u_x))$ and $g_2=\bar{g}_2(x,w_1(x,u,u_x))$  are solutions of the corresponding system in (\ref{sistemas}).
%
\end{proof}

As a consequence of the previous results, 
 the main theorem in this section can now be proved: 

\begin{thm} \label{teofinal}{\rm
Let $(\partial_u,\lambda_1)$ and $(\partial_u,\lambda_2)$ be the canonical representatives of two non-equivalent generalized $\mathcal{C}^{\infty}-$sy\-mme\-tries of  equation (\ref{orden2}) and let 
$\mathcal{A}_1$ and $\mathcal{A}_2$ be their respective equivalence classes. 
Let $f_1,f_2\in\mathcal{C}^{\infty}(M^{(1)})$  be two functions satisfying (\ref{f1f2}) and 
assume that  $g_1,g_2\in\mathcal{C}^{\infty}(M^{(1)})$ satisfy (\ref{sistemas}). Let $h_1,h_2\in  \mathcal{C}^{\infty}(M^{(1)})$ be the functions defined by
\begin{equation}
\label{vfinales} h_i=f_ig_i, \quad (i=1,2),
\end{equation}
and denote  ${\mathbf{Z}_i}=h_i\partial_{u}+\campoA(h_i)\partial_{u_x},$ for $i=1,2.$ The following relations hold:
\begin{enumerate}
\item $(h_1\partial_u,0)\in \mathcal{A}_1$ and $(h_2\partial_u,0)\in \mathcal{A}_2$.
\item The vector fields  $\mathbf{Z_1}$ and $\mathbf{Z}_2$ are  generalized sy\-mme\-tries of (\ref{orden2}). 
\item The system $\{\mathbf{Z_1},\mathbf{Z}_2\}$ is a system of commuting generalized symmetries of $\mathbf{A}$:
\begin{equation}
[\mathbf{Z_1},\mathbf{A}]=[\mathbf{Z}_2,\mathbf{A}]=[\mathbf{Z_1},\mathbf{Z}_2]=0.
\end{equation}
\end{enumerate}
}
\end{thm}
\begin{proof}
With the previous notations,  
by using (\ref{lambdasnew}) and (\ref{sistemas}), 
\begin{equation}\label{cuentas}
\displaystyle{\begin{array}{lllll} \displaystyle\lambda_i-\frac{\mathbf{A}(h_i)}{h_i}
&=&\displaystyle\lambda_i-\frac{\mathbf{A}(f_ig_i)}{f_ig_i}&=&\displaystyle\lambda_i-\frac{\mathbf{A}(f_i)}{f_i}-\frac{\mathbf{A}(g_i)}{g_i}\\
&=&\displaystyle\rho_i-\frac{\mathbf{A}(g_i)}{g_i}&=&0.
\end{array}}
\end{equation} 
This proves that, for $i=1,2,$ the pair $(h_i\partial_u,0)$ is a generalized $\mathcal{C}^{\infty}-$sy\-mme\-try of (\ref{orden2}), which is $\mathbf{A}-$equi\-va\-lent to $(\partial_u,\lambda_i),$ according to Definition \ref{defeq}. In fact, by denoting  $\mathbf{Z}_i=h_i\partial_{u}+\campoA(h_i)\partial_{u_x}$, for $i=1,2$, the vector fields $\mathbf{Z}_1$ and $\mathbf{Z}_2$ are generalized symmetries of (\ref{orden2}).  Relations (\ref{cuentas}) imply

\begin{equation}
\label{losZ}
\mathbf{Z}_i=h_i\partial_{u}+\campoA(h_i)\partial_{u_x}=g_i\mathbf{Y}_i= h_i\mathbf{X}_i,
\end{equation} for $i=1,2.$ 

By using (\ref{corchetesY}), (\ref{cuentas}), (\ref{losZ}) and the properties of the Lie bracket, it follows
$$\displaystyle{\begin{array}{l} \displaystyle[\mathbf{Z}_i,\campoA]
=\displaystyle[h_i\mathbf{X}_i,\campoA]= h_i[\mathbf{X}_i,\campoA]-\campoA(h_i)\mathbf{X}_i
=h_i\lambda_i\mathbf{X}_i-\campoA(h_i)\mathbf{X}_i=0,\end{array}}$$
for $i=1,2.$
 Finally, by (\ref{corchetesnuevos}), (\ref{sistemas}) and (\ref{losZ})
$$[\mathbf{Z}_1,\mathbf{Z}_2]=[g_1\Zu,g_2\Zd] =g_1g_2[\Zu,\Zd]+g_1\Zu(g_2)\cdot\Zd-g_2\Zd(g_1)\cdot \Zu=0.$$

\end{proof}

The results obtained in this section will now be used  in the following example to construct two commuting generalized symmetries from the $\mathcal{C}^{\infty}-$sym\-me\-tries (\ref{lambdasymeje}) of equation (\ref{painleve}).

\begin{example}\label{eje3}
{\rm The studies
 of equation (\ref{painleve}) made in Examples \ref{eje1} and \ref{eje2} will be  continued in this example. Lemma \ref{lema8} will be used to construct two functions $g_1$ and $g_2$ satisfying (\ref{sistemas}). In terms of the variables $\{x,w_1,w_2\}$ where $w_1$ and $w_2$ are given by (\ref{wejemplo}), the required functions $g_1=g_1(x,w_2)$ and $g_2=g_2(x,w_1)$ are solutions of the equations (\ref{pdered}):
\begin{equation}
\label{pderedeje}
\begin{array}{l}
({g}_1)_x-\left({w_2}^{2}-\dfrac{1}{2}\,\bigl(F(x)  -1\bigr)\right)({g}_1)_{w_2}=-2w_2{g}_1,\\
({g}_2)_x+\left({w_1}^{2}-\dfrac{1}{2}\,\bigl(F(x)  +1\bigr)\right)({g}_2)_{w_1}=2w_1{g}_2.
\end{array}
\end{equation}

In what follows,  explicit expressions of some solutions of (\ref{pderedeje}) will be obtained in  terms of solutions of the $2$nd-order linear equations
\begin{equation}
\label{lineal}
\psi_{xx}=\dfrac{1}{2}\,\bigl(F(x) +1\bigr)\psi,\qquad \theta_{xx}=\displaystyle{\dfrac{1}{2}\,\bigl(F(x) -1\bigr)}\theta.
\end{equation}
These are the linear equations obtained from the  Riccati-type equations (\ref{reduceje}) by means of the standard transformations  $w_1=-\psi'/\psi$ and $w_2=\theta'/\theta,$  respectively. Let the pairs  $\psi_1,\psi_2$ and  $\theta_1,\theta_2$ be linearly independent solutions of the respective equations given in (\ref{lineal}). Their corresponding Wronskians will be denoted by $W_1=W(\psi_1,\psi_2)=\psi_1\psi_2'-\psi_1'\psi_2$ and $W_2=W(\theta_1,\theta_2)=\theta_1\theta_2'-\theta_1'\theta_2$. It can be checked that the functions
\begin{equation}\label{gtilde}
{g}_1(x,w_2)=\dfrac{1}{W_2} \left( \theta_2 w_2-\theta_2'  \right) ^{2}
\quad \mbox{and} \quad
{g}_2(x,w_1)=\dfrac{1}{W_1}\left( \psi_2 w_1+\psi_2'  \right) ^{2}
\end{equation} satisfy the corresponding equation in (\ref{pderedeje}). By using (\ref{wejemplo}) and the functions $f_1$ and $f_2$ obtained in (\ref{f1f2ej}) the functions $h_1$ and $h_2$ given by (\ref{vfinales}) become
\begin{equation}
\label{heje}
h_1=\dfrac{1}{4W_2} \left((  u_x+{u}^{2}-1) \theta_2-2u\theta_2'  \right) ^{2}
,\quad
h_2=\dfrac{1}{4W_1} \left((  u_x+{u}^{2}+1) \psi_2-2u\psi_2'  \right) ^{2}
.
\end{equation} 
Finally, according to (\ref{losZ}), two commuting generalized symmetries for equation (\ref{painleve}) are given by

$$\begin{array}{l}
\mathbf{Z}_1=\dfrac{1}{4W_2} \left((  u_x+{u}^{2}-1) \theta_2-2u\theta_2'  \right) ^{2}
\left(\partial_u+\left({\dfrac {u_x}{u_{{}}}}-u_{{}}+\dfrac{1}{u}\right)\partial_{u_x}\right),\\
\mathbf{Z}_2=\dfrac{1}{4W_1} \left((  u_x+{u}^{2}+1) \psi_2-2u\psi_2'  \right) ^{2}
\left(\partial_u+\left({\dfrac {u_x}{u_{{}}}}-u_{{}}-\dfrac{1}{u}\right)\partial_{u_x}\right).
\end{array}
$$
} \hfill\qedsymbol \end{example}

In the next section it will be shown that  the system $\{\mathbf{Z_1},\mathbf{Z}_2\}$ of commuting symmetries constructed from two non-equivalent $\mathcal{C}^{\infty}-$symmetries of equation (\ref{orden2}) permits the integration of the equation by using quadratures alone.

\section{Generalized $\mathcal{C}^{\infty}-$sy\-mme\-tries and integrability by quadratures}\label{s5}

As a consequence of Theorem \ref{teofinal}, two non-equivalent generalized $\mathcal{C}^{\infty}-$sy\-mme\-tries of  equation (\ref{orden2}) can be used  to construct a
system of commuting symmetries of $\campoA.$ Next,  this system is used to compute \textbf{by quadratures}  first integrals of the equation (\ref{orden2}) associated to the
generalized $\mathcal{C}^{\infty}-$sy\-mme\-tries. The idea is similar to the one used in the proof of Proposition \ref{rectificacion1}: to construct the first integrals $I_1=I_1(x,u,u_x)$ and $I_1=I_1(x,u,u_x)$ satisfying respectively the systems

\begin{equation}
\label{sistemasnuevos}
\begin{array}{ll}
\left\{\begin{array}{l}\mathbf{A}(I_1) =0, \\ \mathbf{Z_1}(I_1)=0, \\ \mathbf{Z}_2(I_1)=1, \end{array}\right.
& \left\{\begin{array}{l}\mathbf{A}(I_2) =0, \\ \mathbf{Z_1}(I_2)=1, \\ \mathbf{Z_2}(I_2)=0.\end{array}\right.
\end{array}
\end{equation}
By Fröbenius theorem, both system are compatible. By using the expressions $\mathbf{A}=\partial_x+u_x\partial_u+\phi\partial_{u_x}$, $\mathbf{Z}_1=(f_1g_1)\partial_u+\campoA(f_1g_1)\partial_{u_x}$, 
$\mathbf{Z}_2=(f_2g_2)\partial_u+\campoA(f_2g_2)\partial_{u_x}$, $\mathbf{A}(f_1g_1)=\lambda_1(f_1g_1)$ and  $\mathbf{A}(f_2g_2)=\lambda_2(f_2g_2)$,
it can be chec\-ked that  such functions $I_1$ and $I_2$ must satisfy

\begin{equation}\label{I1}
{(I_1)}_x=\dfrac{\lambda_1u_x-\phi}{f_2g_2(\lambda_2-\lambda_1)},\quad {(I_1)}_u=\dfrac{-\lambda_1}{f_2g_2(\lambda_2-\lambda_1)},\quad
{(I_1)}_{u_x}=\dfrac{1}{f_2g_2(\lambda_2-\lambda_1)},
\end{equation}
and
\begin{equation}\label{I2} {(I_2)}_x=\dfrac{\lambda_2u_x-\phi}{f_1g_1(\lambda_1-\lambda_2)},\quad
{(I_2)}_u=\dfrac{-\lambda_2}{f_1g_1(\lambda_1-\lambda_2)},\quad {(I_2)}_{u_x}=\dfrac{1}{f_1g_1(\lambda_1-\lambda_2)}.
\end{equation}

Observe that the existence of the function $I_1$  satisfying (\ref{sistemasnuevos}) could also be proved by checking  that the mixed derivatives of the functions in the right hand sides of (\ref{I1}) coincide:
$$
\label{mixed}
\begin{array}{ll}
\left(\dfrac{\lambda_1u_x-\phi}{f_2g_2(\lambda_2-\lambda_1)}\right)_u=\left(\dfrac{-\lambda_1}{f_2g_2(\lambda_2-\lambda_1)}\right)_x, &  \left(\dfrac{\lambda_1u_x-\phi}{f_2g_2(\lambda_2-\lambda_1)}\right)_{u_x}=\left(\dfrac{1}{f_2g_2(\lambda_2-\lambda_1)}\right)_x, \\ \\ \left(\dfrac{-\lambda_1}{f_2g_2(\lambda_1-\lambda_2)}\right)_{u_x}=\left(\dfrac{1}{f_2g_2(\lambda_2-\lambda_1)}\right)_u.  
\end{array}
$$ 
Similarly,  the existence of the function $I_2$  satisfying (\ref{sistemasnuevos}) could also be proved by checking that the mixed derivatives of the functions in the right hand sides of (\ref{I2}) do also coincide.
Therefore the functions $I_1$  and $I_2$ that satisfy (\ref{sistemasnuevos})  can locally be defined  by quadratures 
from (\ref{I1}) and (\ref{I2}). Since by (\ref{losZ}) $\mathbf{Z}_i=f_ig_i\mathbf{X}_i$, the function $I_i$ is a first integral of $\mathbf{X}_i$, for $i=1,2$. In other words, $I_1$ and $I_2$ are first integrals of $\campoA$ associated to $(\partial_u,\lambda_1)$ and $(\partial_u,\lambda_2)$ respectively.

Consequently, the following theorem holds:

\begin{thm}
\label{simetrias} {\rm Let $(\partial_u,\lambda_1)$ and $(\partial_u,\lambda_2)$ be the canonical representatives of two non-equivalent generalized $\mathcal{C}^{\infty}-$sy\-mme\-tries of equation
(\ref{orden2}). Let $f_1,f_2$ (resp. $g_1,g_2$) be some
pairs of functions in $\mathcal{C}^{\infty}(M^{(1)})$ satisfying (\ref{f1f2}) (resp. (\ref{sistemas})).
Two functionally independent first integrals $I_1$ and $I_2$ of $\campoA,$ associated to $(\partial_u,\lambda_1)$ and $(\partial_u,\lambda_2),$ respectively, can be
found by quadratures from (\ref{I1}) and (\ref{I2}).
}
\end{thm}

\begin{example}
\label{eje4}
{\rm This example is a continuation of Examples \ref{eje1}, \ref{eje2} and \ref{eje3}.  Theorem \ref{simetrias} will be  applied to integrate by quadratures equation (\ref{painleve}). The functions $h_1=f_1g_1$ and $h_2=f_2g_2$ obtained in (\ref{heje}) let the construction of systems (\ref{I1})-(\ref{I2}). Two first integrals $I_1$ and $I_2$ for this equation can be obtained  by quadratures by using (\ref{I1}) and (\ref{I2}) respectively: 
\begin{equation}
\label{ipejegeneral}
I_1=\frac{ ({u}^{2}+u_x+1)\,\psi_1 -2u\psi_1'}{ ({u}^{2}+u_x+1)\psi_2 -2u\psi_2'}
\quad \mbox{and} \quad
I_2=\frac{ ({u}^{2}+u_x-1)\,\theta_1 -2u\theta_1'}{ ({u}^{2}+u_x-1)\theta_2 -2u\theta_2'}.
\end{equation}
These first integrals provide the general solution of equation (\ref{painleve}) in terms of the solutions of the corresponding linear equations (\ref{lineal}):
\begin{equation}\label{generalsol}
u_{{}}(x)=\left( {\frac {C_1\,\psi_2'
  -\psi_1'
}{C_1\,\psi_2  -\psi_1  }}-{\frac {
C_2\,\theta_2'  -\theta_1'  }{C_2\,\theta_2
  -\theta_1  }} \right)^{-1} \quad (C_1,C_2\in \mathbb{R}).
\end{equation}

In \cite{ince} the general solution of equation (\ref{painleve}) is expressed in terms of the solutions of a third-order nonlinear ODE that becomes a fourth-order linear ODE under differentiation. Observe that in this paper  the general solution (\ref{generalsol}) has been obtained in terms of two fundamental sets of solutions for two second-order linear ODEs, although equation (\ref{painleve}) is not linearisable. }
\hfill\qedsymbol
\end{example}

\section{Integrating factors and Jacobi last multipliers  without additional integration}\label{secJacobi}

In this section, it is shown  how several classical objects (as Jacobi last multipliers and integrating factors) associated to the given ODE, can be considered as by-products of the procedure of integration by quadratures described in the previous section.  

\subsection{Integrating factors and a Jacobi last multiplier for equation (\ref{orden2})}


Some relations between first integrals of $2$nd-order ODEs and $\mathcal{C}^\infty-$sy\-mme\-tries have been established in \cite{ muriel2009,muriel09wascom,muriel08}: if $I\in \mathcal{C}(M^{(1)})$ is a first integral of $\campoA$ associated to a $\mathcal{C}^\infty-$symmetry whose canonical representative is
$(\partial_u,\lambda),$ then $\mu=I_{u_x}$ is an integrating factor of equation (\ref{orden2}) such that $\mu(u_{xx}-\phi)=\mathbf{D}_x(I)$ and the following identities hold:
\begin{equation}\label{Iuna}
I_x=\mu(\lambda u_x-\phi),\quad I_u=-\lambda\mu.
\end{equation} 

Suppose that  two non-equivalent $\mathcal{C}^\infty-$symmetries $(\partial_u,\lambda_1)$   and $(\partial_u,\lambda_2)$ of  equation (\ref{orden2}) are known, and let  $f_1,f_2$ (resp. $g_1,g_2$) be some
pairs of functions in $\mathcal{C}^{\infty}(M^{(1)})$ satisfying (\ref{f1f2}) (resp. (\ref{sistemas})). By (\ref{I1})-(\ref{I2}), \begin{equation}\label{factint}
\mu_1=\dfrac{1}{f_2g_2(\lambda_2-\lambda_1)}, \qquad \mu_2=\dfrac{1}{f_1g_1(\lambda_1-\lambda_2)}
\end{equation}
are two integrating factors of (\ref{orden2}). Observe that the two first equations in (\ref{I1}) (resp. (\ref{I2})) are the
equations (\ref{Iuna}) for $\mu=\mu_1$ (resp. $\mu=\mu_2$)  given in (\ref{factint}).



Another consequence of Theorem \ref{teofinal} is that  the system of commuting symmetries can be used  to derive an explicit expression of a Jacobi last multiplier of the
equation. According to \cite{whittaker1988treatise} and with the same notation used in the previous sections, the reciprocal of the determinant
$$\left|
\begin{array}{ccc}
0 & f_1{g_1} & \A(f_1{g_1})\\
0 & f_2{g_2} & \A(f_2{g_2})\\
1&u_{x}&\phi
\end{array}\right|=\left|
\begin{array}{ccc}
0 & f_1{g_1} & {\lambda_1}f_1{g_1}\\
0 &f_2 {g_2} & {\lambda_2}f_2{g_2}\\
1&u_{x}&\phi
\end{array}\right|=f_1{g_1}f_2{g_2}({\lambda_2}-{\lambda_1})$$ is a Jacobi last multiplier of  equation (\ref{orden2}). The   earlier  discussion proves the next theorem:
\begin{thm}\label{ifJM}
{\rm Let $(\partial_u,\lambda_1)$ and $(\partial_u,\lambda_2)$ be the canonical representatives of two non-equivalent generalized $\mathcal{C}^{\infty}-$sy\-mme\-tries of equation (\ref{orden2}). Let $f_1,f_2$ (resp. $g_1,g_2$) be some pairs of functions in
$\mathcal{C}^{\infty}(M^{(1)})$ satisfying (\ref{f1f2}) (resp. (\ref{sistemas})). Then:
\begin{enumerate}
\item The function $\mu_{1}=\mu_{1}(x,u,u_{x})$ (resp. $\mu_{2}=\mu_{2}(x,u,u_{x})$) given in (\ref{factint}) is an integrating factor of the equation (\ref{orden2}). A first integral $I_1$ (resp. $I_2$) of $\campoA$ associated to $(\partial_u,\lambda_1)$ (resp.  $(\partial_u,\lambda_2)$) can be obtained by quadratures from   (\ref{I1}) (resp. (\ref{I2})).
\item The function $M=M(x,u,u_{x})$ defined by
\begin{equation}\label{JLM}
M=\dfrac{1}{ f_1{g_1}f_2{g_2}({\lambda_2}-{\lambda_1}) }
\end{equation} is a Jacobi last multiplier of the equation (\ref{orden2}).\end{enumerate}}
\end{thm}

\begin{example}
\label{eje5} {\rm By continuing  the Examples \ref{eje1}$-$\ref{eje4}, Theorem \ref{ifJM} implies that two integrating factors of equation (\ref{painleve}) can be calculated   as in (\ref{factint}) by using the functions given in (\ref{lambdasymeje}), (\ref{f1f2ej}) and (\ref{gtilde}):  
\begin{equation}
\label{fiejegeneral}
\mu_1=\frac {-2uW_1}{ \left( 
 \left(u^{2}+u_x+1\right)\psi_2-2u\psi_2'
 \right) ^{2}}
 ,\qquad
 \mu_2=\frac {-2uW_2}{ \left( 
  \left(u^{2}+u_x-1\right)\theta_2-2u\theta_2'
  \right) ^{2}}
\end{equation} expressed in terms of the solutions of the corresponding linear equations (\ref{lineal}).  A Jacobi last multiplier of the equation (\ref{painleve}) can be obtained by using (\ref{JLM}):

$$M=\dfrac{8uW_1W_2}{\left( 
 \left(u^{2}+u_x+1\right)\psi_2-2u\psi_2'
 \right) ^{2}   \left( 
   \left(u^{2}+u_x-1\right)\theta_2-2u\theta_2'
   \right) ^{2} } $$


\hfill\qedsymbol} 
\end{example}

\subsection{Integrating factors of the reduced  equations associated to the $\mathcal{C}^{\infty}-$sy\-mme\-tries}\label{s7}

 Theorem \ref{simetrias}  provides a novel procedure that simultaneously uses two known generalized $\mathcal{C}^{\infty}-$sy\-mme\-tries of (\ref{orden2}) to construct, {by quadratures}, two independent first integrals, $I_1$ and $I_2$ of $\campoA.$  
 In this section it is shown  how any
pair of functions $g_1,g_2$ in $\mathcal{C}^{\infty}(M^{(1)})$   satisfying (\ref{sistemas}) can be used to provide integrating factors of the reduced  equations associated to the given  generalized $\mathcal{C}^{\infty}-$sy\-mme\-tries.

Consider the system of coordinates $\{x,w_1,w_2\}$ constructed in Proposition \ref{rectificacion1} and the reduced equations (\ref{reducidas}) associated to the given $\mathcal{C}^{\infty}-$symmetries.  Recall that a function $\mu=\mu(x,u)$ is an integrating factor of a first-order ODE $u_x=\varphi(x,u)$ if $\mu_x+(\varphi\mu)_u=0$ \cite{blumanlibro, blumanancoarticulo}; in other words, $\mu$ is an integrating factor of that equation  if \begin{equation}
\label{fiorden1} \campoA_0(\mu)=-\mu\, \varphi_{u},
\end{equation} where $\campoA_0=\partial_x+\varphi\partial_u$ denotes the corresponding vector field.
According to (\ref{pdered}) and taking (\ref{camposreducidas}) into account, 
 $\campoA_1({g}_2)={g}_2({{\phi}_1})_{w_1}$ (resp. $\campoA_2({g}_1) ={g}_1({{\phi}_2})_{w_2}$). Therefore 
$$\campoA_1\left(\frac{1}{{g}_2}\right)
=-\frac{{g}_2\,({{\phi}_1})_{w_1}}{ {({g}_2)}^{2}}=-\frac{({{\phi}_1})_{w_1}}{{g}_2} \quad \mbox{and} \quad
\campoA_2\left(\frac{1}{{g}_1}\right)=-\frac{({{\phi}_2})_{w_2}}{g_1}.$$ This implies, by using (\ref{fiorden1}), that the reciprocal of the functions $g_1$ and $g_2$  provide integrating factors of the reduced equations (\ref{reducidas}):

\begin{thm}\label{teofiredu}
{\rm Let $(\partial_u,\lambda_1)$ and $(\partial_u,\lambda_2)$ be the canonical representatives of two non-equivalent generalized $\mathcal{C}^{\infty}-$sy\-mme\-tries of equation (\ref{orden2}). Let $\{x,w_1,w_2\}$ be the local system of coordinates of
$M^{(1)}$ given in Proposition 
\ref{rectificacion1} and let ${g}_1={g}_1(x,w_2)$ and ${g}_2={g}_2(x,w_1)$ be  particular solutions of the corresponding equation in
(\ref{pdered}). Then 
\begin{equation} \label{ifredu}
{\nu}_1(x,w_1)=\frac{1}{{g}_2(x,w_1)} \quad \mbox{and} \quad
\,
{\nu}_2(x,w_2)=\frac{1}{{g}_1(x,w_2)} 
\end{equation} are respectively  integrating factors of the two reduced equations  that appear in (\ref{reducidas}).
}
\end{thm}

\subsection{Integrating factors of the  auxiliary equations associated to the $\mathcal{C}^{\infty}-$sy\-mme\-tries}\label{s7bis}

Suppose that   $\widehat{I_1}=\widehat{I_1}(x,w_1)$ and $\widehat{I_2}=\widehat{I_2}(x,w_2)$
denote the respective first integrals of the reduced equations (\ref{reducidas}) associated to the integrating factors (\ref{ifredu}) .
Then \begin{equation} \label{solred}
\widehat{I_i}(x,w_i)=C_i, \quad (i=1,2) 
\end{equation} provide  the general solutions
of the respective reduced equations in (\ref{reducidas}), where $C_1, C_2\in\mathbb{R}.$ The corresponding auxiliary equations are obtained by writing $w_i$ 
in terms of $(x,u,u_x)$ in (\ref{solred}):
\begin{equation} \label{auxiliaresimplicitas}
\widehat{I_i}\bigl(x,w_i(x,u,u_x)\bigr)=C_i, \quad (i=1,2).
\end{equation} 
Locally (\ref{auxiliaresimplicitas}) can be written in the form
\begin{equation}\label{auxiliary}
u_x=H_i(x,u,C_i), \quad (i=1,2). 
\end{equation} 

Now it is shown that some of the previous results can be used to obtain integrating factors of the equations in (\ref{auxiliary}). To this end, recall that according to Corollary 6 in \cite{muriel2014Jacobi}, any given Jacobi last multiplier of (\ref{orden2}) provides an integrating factor of the auxiliary equation that
appears in the reduction process associated to a given $\lambda-$symmetry. In the following theorem  the Jacobi last multiplier (\ref{JLM}) is used to find integrating
factors of the auxiliary equations (\ref{auxiliary}).
\begin{thm}\label{teofiaux}
{\rm Let $(\partial_u,\lambda_1)$ and $(\partial_u,\lambda_2)$ be the canonical representatives of two non-equivalent generalized $\mathcal{C}^{\infty}-$sy\-mme\-tries of equation (\ref{orden2}). Let $f_1,f_2$ (resp. $g_1,g_2$) be some  functions in
$\mathcal{C}^{\infty}(M^{(1)})$ satisfying (\ref{f1f2}) (resp. (\ref{sistemas})).
Then the function 
\begin{equation}
\label{ifauxiliar}
\widetilde{\nu}_i(x,u)=\frac{1}{f_i\bigl(x,u,H_i(x,u,C_i)\bigr)\,g_i\bigl(x,u,H_i(x,u,C_i)\bigr)}, \qquad (i=1,2) 
\end{equation} 
is an integrating factor of the corresponding auxiliary equation (\ref{auxiliary}). 
}
\end{thm}
\begin{proof}
A function  $g_1$ (resp. $g_2$) satisfying (\ref{sistemas}), written in terms of  $\{x,w_2\}$  (resp. $\{x,w_1\}$),  is a solution of the first (resp. the second) equation in (\ref{pdered}). By Theorem
\ref{teofiredu} the corresponding functions (\ref{ifredu}) are integrating factors of the respective equations in (\ref{reducidas}). The respective first integrals    $\widehat{I_1}=\widehat{I_1}(x,w_1)$ and $\widehat{I_2}=\widehat{I_2}(x,w_2),$
satisfy
\begin{equation}\label{I1m1}
(\widehat{I_1})_{w_1}=\frac{1}{{g}_2(x,w_1)}, \qquad 
(\widehat{I_2})_{w_2}=\frac{1}{{g}_1(x,w_2)}.
\end{equation}
For $i=1,2,$ let $I_i=I_i(x,u,u_x)$ 
denote the function $\widehat{I_i}$ 
when $w_i$ 
is expressed in terms of $(x,u,u_x).$ 
Let $M$ be the Jacobi last multiplier of equation (\ref{orden2}) given in (\ref{JLM}). According to Corollary 6 in \cite{muriel2014Jacobi}, the restriction of the function 
\begin{equation}\label{inter}
\frac{M}{({I_1})_{u_x}}=\frac{1}{ f_1{g_1}f_2{g_2}({\lambda_2}-{\lambda_1}) ({I_1})_{u_x}} 
\end{equation} to
$\Delta_1=\{(x,u,u_x)\in M^{(1)}:u_x=H_1(x,u,C_1)\} $
is an integrating function of the auxiliary equation $u_x=H_1(x,u,C_1).$ 
By (\ref{m1}) and by taking (\ref{I1m1}) into account, 
\begin{equation}\label{I1ux}
({I_1})_{u_x}=-\frac{1}{f_2g_2(\lambda_1-\lambda_2)}
\end{equation} and (\ref{inter}) becomes

\begin{equation}\label{inter2}
\begin{array}{lll}
\dfrac{M}{({I_1})_{u_x}} =\dfrac{1}{f_1g_1}. 
\end{array}
\end{equation}
Therefore the restriction of (\ref{inter2}) to $\Delta_1$ gives us the integrating factor $\widetilde{\nu}_1$ given in (\ref{ifauxiliar}). A similar reasoning by
using the Jacobi last multiplier $-M,$ where $M$ is defined in (\ref{JLM}), proves that the function $\widetilde{\nu}_2$ given in (\ref{ifauxiliar}) is an integrating
factor of the auxiliary equation $u_x=H_2(x,u,C_2),C_2\in \mathbb{R}.$
\end{proof}

%

\begin{example}\label{eje6}
{\rm  As a continuation of Example \ref{eje5}, by Theorem \ref{teofiredu} the reciprocal of the functions (\ref{gtilde}), i.e. the functions 
 $${\nu}_1(x,w_1)=\dfrac{W_1}{\left( \psi_2 w_1+\psi_2'  \right) ^{2}
   }\quad \mbox{and} \quad
{\nu}_2(x,w_2)=\dfrac
  {W_2 
  }{ \left( \theta_2 w_2-\theta_2'  \right) ^{2}},
 $$
are integrating factors of the reduced equations (\ref{reduceje}). The associated first integrals become
$$ \widehat{I}_1(x,w_1)=\frac { w_1\,\psi_1 +
 \psi_1' }{ w_1\psi_2 + 
 \psi_2' }
 \quad \mbox{and} \quad
 \widehat{I}_2(x,w_2)=\frac {w_2\,\theta_1 -
 \theta_1' }{w_2\theta_2 -{ 
 \theta_2'}}.
$$
 By using (\ref{wejemplo}), the auxiliary equations (\ref{auxiliary}) become :
 \begin{equation}
 \label{auxiliaryejemplo}
 \begin{array}{l}
 u_x=2u\left({\dfrac {C_1\,\psi_2'
    -\psi_1'
  }{C_1\,\psi_2  -\psi_1  }}\right)-u^2-1\quad \mbox{and} \quad
  
 u_x= 2u\left({\dfrac {
  C_2\,\theta_2'  -\theta_1'  }{C_2\,\theta_2
    -\theta_1  }}\right)-u^2+1.
 \end{array}
 \end{equation}
 By Theorem \ref{teofiaux}, the functions (\ref{ifauxiliar}) constructed by using (\ref{f1f2ej}) and (\ref{gtilde}):
$$ \tilde{\nu}_1(x,u)=-\dfrac{C_1\,\psi_2  -\psi_1 }{u^2({C_1\,\psi_2'
    -\psi_1'
  })} \quad \mbox{and} \quad  
  \tilde{\nu}_2(x,u)=\dfrac{C_2\,\theta_2
    -\theta_1 }{u^2(C_2\,\theta_2'  -\theta_1')} 
$$
 are integrating factors of the respective auxiliary equation in (\ref{auxiliaryejemplo}). The corresponding first integrals are the functions $I_1$ and $I_2$ derived in (\ref{fiejegeneral}).
 \hfill \qedsymbol }
\end{example}

\subsection{Scheme of the methods}\label{secesquema}

 Figure \ref{diagrama} presents a scheme of the alternative procedures that can be followed to integrate a second-order equation admitting two generalized 
$\mathcal{C}^{\infty}-$sym\-me\-tries.  Their canonical representatives $(\partial_u,\lambda_1)$ and $(\partial_u,\lambda_2)$ will be used in the scheme.  

\begin{enumerate}
\item In the central branches of the diagram,  it is  sketched the use of Theorem \ref{simetrias} to calculate two independent first integrals $I_1,I_2$ by {quadratures} through systems (\ref{I1}) and (\ref{I2}). The associated integrating factors are given in (\ref{factint}). 

\item The left branch of the figure solves the ODE by using the reduced and auxiliary equations associated to the first $\mathcal{C}^{\infty}-$symmetry. The reduced equation can be solved by {quadrature} by using the integrating factor ${\nu}_1$ in Eq. (\ref{ifredu})   (Theorem \ref{teofiredu}). The auxiliary equation can be solved by {quadrature} by using the integrating factor $\widetilde{\nu}_1$ given in (\ref{ifauxiliar}) (Theorem \ref{teofiaux}). A similar process can be followed by using the second  $\mathcal{C}^{\infty}-$symmetry and the corresponding functions ${\nu}_2$ and $\widetilde{\nu}_2$ (right branch of the scheme).
\item Finally, the functions  ${\nu}_1$ and ${\nu}_2$ given in (\ref{ifredu}) let the determination by {quadrature} of a first integral for each  reduced equation  (Theorem \ref{teofiaux}). The general solution of the ODE arises by following the dashed arrows in the diagram. 
\end{enumerate}


\begin{figure}[H]
\tikzstyle{decision} = [diamond, draw, fill=blue!20, text width=4.5em, text badly centered, node distance=2.5cm, inner sep=0pt] \tikzstyle{block} = [rectangle, draw,
fill=blue!20, text centered, rounded corners] \tikzstyle{line} = [draw, thick, color=black, -latex'] \tikzstyle{join} = [draw, very thick, color=black,
length=1.5cm,-latex'] \tikzstyle{cloud} = [draw, ellipse,fill=red!20, node distance=1.5cm, minimum height=2em] \tikzstyle{myarrow}=[->, >=triangle 90, thick]
\begin{tikzpicture}[xscale=0.75, yscale=0.85, transform shape, node distance = 1.5cm, auto]
\node at (0,0)[block, rectangle split, rectangle split parts=2] (init){Original Equation \nodepart{second}{$u_{xx}=\phi(x,u,u_x)$}};

\node [block, below of=init, left of=init, xshift=-3.5cm] (lambda1) {$\mathbf{X}_1=\partial_u^{[\lambda_1,(1)]}$}; \node [block, below of=init, right of=init,
xshift=3.5cm] (lambda2) {$\mathbf{X}_2=\partial_u^{[\lambda_2,(1)]}$}; \node [block, below of=lambda1,yshift=-0.5cm] (reducida1) {${w_1}_x=\phi_1(x,w_1)$}; \node [block,
below of=lambda2,yshift=-0.5cm] (reducida2) {${w_2}_x=\phi_2(x,w_2)$}; \node [block, below of=reducida1,rectangle split, rectangle split parts=2,yshift=-1cm] (I1)
{$I_1(x,u,u_x)=C_1$\nodepart{second}{$u_x=H_1(x,u,C_1)$}}; \node [block, below of=reducida2,rectangle split, rectangle split parts=2,yshift=-1cm] (I2)
{$I_2(x,u,u_x)=C_2$\nodepart{second}{$u_x=H_2(x,u,C_2)$}}; \node [block, below of=init] (rho)
{$\rho=\frac{\mathbf{X}_1(\lambda_2)-\mathbf{X}_2(\lambda_1)}{\lambda_1-\lambda_2}$}; 
\node [block, below of=rho, right of=rho, xshift=-0.45cm,yshift=-0.5cm]
(Z1) {$\mathbf{Y}_1=f_1\Yu$}; 
\node[block, below of=rho, left of=rho, xshift=0.45cm,yshift=-0.5cm] (Z2) {$\mathbf{Y}_2=f_2\Yd$}; \node[block, below of=Z2,xshift=-0.25cm, yshift=-1cm] (V2)
{$\mathbf{Z}_2=f_2g_2\Yd$}; \node[block, below of=Z1,xshift=0.25cm,yshift=-1cm] (V1) {$\mathbf{Z}_1=f_1g_1\Yu$}; \node [block, below of=V2, left of=V2, xshift=3cm] (w1w2)
{$dI_1,dI_2$}; \node [block, below of=w1w2,rectangle split, rectangle split parts=2] (sol) {General Solution\nodepart{second}{$I_1=C_1,I_2=C_2$}}; \node
[block, below of=I1,yshift=-1.5cm,rectangle split, rectangle split parts=2] (solAUX1) {General
Solution\nodepart{second}{$\widetilde{I_2}(x,u,C_1)=\widetilde{C_2}$}}; \node [block, below of=I2,yshift=-1.5cm,rectangle split, rectangle split parts=2]
(solAUX2) {General Solution\nodepart{second}{$\widetilde{I_1}(x,u,C_2)=\widetilde{C_1}$}};
\path [line] (init) -- node [left, xshift=-0.25cm, yshift=0.25cm, color=black] {$(\mathbf{v}_1,\lambda_1)$} (lambda1);
\path [line] (init) -- node [right, xshift=0.25cm, yshift=0.25cm, color=black] {$(\mathbf{v}_2,\lambda_2)$} (lambda2); \path [line] (lambda1) -- node [left, color=black]
{$\mathbf{X}_1(w_1)=0$} (reducida1); \path [line] (lambda2) -- node [right, color=black] {$\mathbf{X}_2(w_2)=0$} (reducida2); \path [line] (reducida1) -- node [left,
color=black] { \begin{tabular}{c} I.F. ${\nu}_1$ \\ Eq. (\ref{ifredu})
\end{tabular}} (I1);
\path [line] (reducida2) -- node [right, color=black] {
\begin{tabular}{c}
I.F. ${\nu}_2$ \\ Eq. (\ref{ifredu})
\end{tabular}
} (I2); \path [line] (I1) -- node [left,xshift=-0.25cm, color=black] {
\begin{tabular}{c}
I.F. $\widetilde{\nu}_1$ \\ Eq. (\ref{ifauxiliar})
\end{tabular}
} (solAUX1); \path [line] (I2) --node [right, xshift=0.25cm, color=black] {
\begin{tabular}{c}
I.F. $\widetilde{\nu}_2$ \\ Eq. (\ref{ifauxiliar})
\end{tabular}
} (solAUX2); \path [line] (lambda1.east) -- (rho.west); \path [line] (lambda2.west) -- (rho.east); \path [line] (rho) -- node [right, xshift=0.2cm, yshift=0.1cm, color=black]
{$\Yd(f_1)=\rho\,f_1$} (Z1); \path [line] (rho) -- node [left, xshift=-0.2cm, yshift=0.1cm, color=black] {$\Yu(f_2)=\rho\,f_2$} (Z2); 
\path [line] (Z2.west) -- node [above]
{$(\partial_{w_1},\rho_2)$} (reducida1.east); 
\path [line] (Z1.east) -- node [above] {$(\partial_{w_2},\rho_1)$}(reducida2.west); \path [line] (Z2) -- node [left,
color=black] {$\mathbf{A}(g_2)=\rho_2g_2$} (V2); \path [line] (Z1) -- node [right, color=black] {$\mathbf{A}(g_1)=\rho_1g_1$} (V1); \path [line] (V1) -- (w1w2); \path [line]
(V2) -- (w1w2); \path [line] (w1w2) -- (sol);
\path [line, dashed]
(I1.south) -- (sol.north);\path [line,dashed]
(I2.south) -- (sol.north);

\draw[->,>=triangle 90] (init.west) -| ([yshift=0cm, xshift=-0.7cm]reducida1.west) node [left, color=black] { \begin{tabular}{c}
I.F. $\mu_1$ 
\\ Eq. (\ref{factint})
\end{tabular} } |-
(I1.west);
\draw[->, >=triangle 90] (init.east) -| ([yshift=0cm, xshift=3.7cm]reducida2.west) node [right, color=black] { \begin{tabular}{c}
I.F. $\mu_2$ 
\\ Eq. (\ref{factint})
\end{tabular} }|-
(I2.east);
\node[fit=(init) (sol) (I1) (I2), transform shape=false]{};
\end{tikzpicture}
\caption{Integration of a second-order ODE with two non-equivalent $\lambda-$symmetries}
\label{diagrama}
\end{figure}
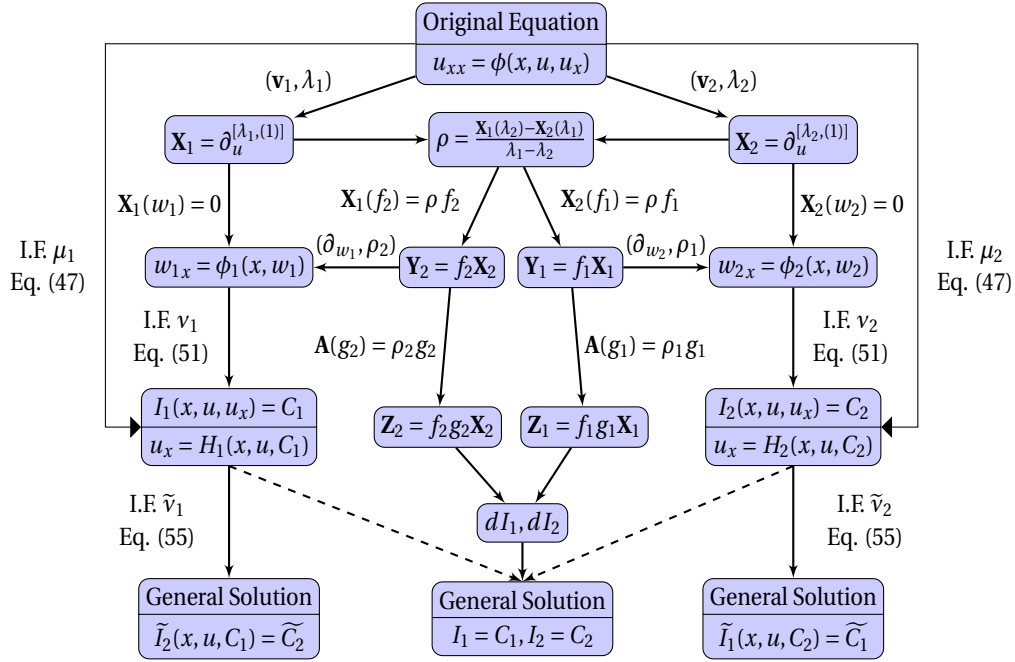


The lateral branches in Figure \ref{diagrama} correspond to the  integration methods derived from the existence of two non-equivalent $\mathcal{C}^\infty-$sym\-me\-tries \cite{muriel01ima1}. The central branches of that figure summarize  the main topic of this paper:   a  schematic  description  of the corresponding procedure will be shown in the next section.

\section{Procedure for integration by quadratures and some examples} \label{sect7}
The following procedure describes the steps that can be followed to obtain the solution by quadratures, according to the method discussed in this paper.

\begin{enumerate}
\item Compute the function $\rho=\dfrac{\mathbf{X}_1(\lambda_2)-\mathbf{X}_2(\lambda_1)}{\lambda_1-\lambda_2},$  where $\mathbf{X}_i=\partial_u+\lambda_i\partial_{u_x},$ for $i=1,2.$
\item Calculate two functions $f_1,f_2$ such that  $\dfrac{\mathbf{X}_1(f_2)}{f_2}= \dfrac{\mathbf{X}_2(f_1)}{f_1}=\rho.$
\item Compute two functions $w_1=w_1(x,u,u_x)$ and $w_2=w_2(x,u,u_x),$ by solving {by quadratures} the systems 

$$\begin{array}{l}
({w_1})_{u}=\dfrac{\lambda_1}{f_2(\lambda_1-\lambda_2)},\qquad ({w_1})_{u_{x}}=-\dfrac{1}{f_2(\lambda_1-\lambda_2)},\\({w_2})_u=\dfrac{\lambda_2}{f_1(\lambda_2-\lambda_1)},\qquad ({w_2})_{u_x}=-\dfrac{1}{f_1(\lambda_2-\lambda_1)}.
\end{array}$$
\item Calculate the function ${\phi}_i(x,w_i)=\campoA(w_i),$ for $i=1,2.$ 
Find a
particular solution $g_1=g_1(x,w_2)$ (resp. $g_2=g_2(x,w_1)$) to the first (resp. second) equation that follows: 
$$ ({{g}_1})_x+({{g}_1})_{w_2}{\phi}_2={{g}_1}\,({{\phi}_2})_{w_2}  \quad  \mbox{and} \quad
({{g}_2})_x+({{g}_2})_{w_1}{\phi}_1= {{g}_2}\,({{\phi}_1})_{w_1}.
$$
\item Use the functions $f_1,f_2$ of step 2 and the functions $g_1,g_2$ of step 4  (expressed in terms of $\{x,u,u_x\}$) to calculate two independent first integrals  $I_1$ and $I_2$ by solving by quadratures the systems 
$$\begin{array}{l}
{I_1}_x=\dfrac{\lambda_1u_x-\phi}{f_2g_2(\lambda_2-\lambda_1)},\quad {I_1}_u=\dfrac{-\lambda_1}{f_2g_2(\lambda_2-\lambda_1)},\quad
{I_1}_{u_x}=\dfrac{1}{f_2g_2(\lambda_2-\lambda_1)},\\

 {I_2}_x=\dfrac{\lambda_2u_x-\phi}{f_1g_1(\lambda_1-\lambda_2)},\quad
{I_2}_u=\dfrac{-\lambda_2}{f_1g_1(\lambda_1-\lambda_2)},\quad {I_2}_{u_x}=\dfrac{1}{f_1g_1(\lambda_1-\lambda_2)}.
\end{array}$$
\end{enumerate}
As an immediate consequence of the above-described procedure the following mathematical objects are obtained:

\begin{itemize}
\item[a)] The general solution of the equation: $I_1=C_1,I_2=C_2,$ where $C_1,C_2\in \mathbb{R}.$
\item[b)] The integrating factors of  equation (\ref{orden2}) given by 
 $$\begin{array}{ll}\mu_1=\dfrac{1}{f_2g_2(\lambda_2-\lambda_1)}, &  \mu_2=\dfrac{1}{f_1g_1(\lambda_1-\lambda_2)}.\end{array}$$
\item[c)] The Jacobi last multiplier of  equation (\ref{orden2}) defined by 
$$M=\dfrac{1}{ f_1{g_1}f_2{g_2}({\lambda_2}-{\lambda_1}) }.$$
\item[d)] Two commuting generalized symmetries of equation (\ref{orden2}):  $$\mathbf{Z}_i=(f_ig_i)\partial_{u}+\campoA(f_ig_i)\partial_{u_x},\qquad (i=1,2).$$
\end{itemize}

\begin{example}
\label{painleveF0}
{\rm Although the general solution (\ref{generalsol}) of any equation of the form (\ref{painleve})  has already been obtained through Examples \ref{eje1}-\ref{eje6}, in this example it is considered the special case  $F(x)=0$ 
\begin{equation}\label{painleve0}
u_{xx}-\frac{1}{2\,u}u_{x}^2+ 2\,u u_{x} +\frac{1}{2}u^{3}+\frac{1}{2\,u}=0,
\end{equation}
for purposes of illustration of the procedure. In this case the computations appointed by  the former procedure can be followed without the aid of a computer; they are relatively simple.  

The three first steps of the method proceed as explained in Examples \ref{eje1}, \ref{eje2} and \ref{eje3}, because the  $\lambda-$symmetries defined by (\ref{lambdasymeje}) do not depend on $F(x):$
\begin{enumerate}
\item[Step 1] $\rho(x,u,u_x)=\dfrac{2}{u}.$
\item[Step 2] $f_1(x,u,u_x)=f_2(x,u,u_x)=u^2.$
\item[Step 3] $ w_1=-{\dfrac {u_x+{u_{{}}}^{2}+1}{2u_{{}}}}\quad \mbox{and} \quad w_2={\dfrac {u_x+{u_{{}}}^{2}-1}{2u_{{}}}}.$
\item[Step 4] Since $\phi_1=\campoA(w_1)={w_1}^{2}-\dfrac{1}{2}$ and $ \phi_2=\campoA(w_2)=-{w_2}^{2}-\dfrac{1}{2},$ two particular solutions of $$({g}_1)_x-\left({w_2}^{2}+\dfrac{1}{2}\right)(g_1)_{w_2}=-2w_2 {g}_1\quad \mbox{and} \quad
({g_2})_x+\left({w_1}^{2}-\dfrac{1}{2}\right)(g_2)_{w_1}=2w_1 {g}_2$$ can be easily found:  ${{g}_1}(x,w_2)=2\,w_2^{2}+1$ and ${{g}_2}(x,w_1)=2\,w_1^{2}-1.$
\item[Step 5] By using the functions $f_1,f_2$ of step 2  and the functions $g_1,g_2$ of step 4, written in variables $\{x,u,u_x\}$ (by means of  the expressions of $w_1$ and $w_2$ given in step 3), the systems that correspond to (\ref{I1}) and (\ref{I2}) are
$$\begin{array}{lll}
{I_1}_x=-\dfrac{1}{2},& {I_1}_u= -{\dfrac {{u_{{}}}^{2}-u_x-1}{ \left( {u_{{}}}^{2}+u_x+1 \right) ^{2}-2\,{u_{{}}}^{2}}},&
{I_1}_{u_x}=-{\dfrac {u_{{}}}{ \left(
{u_{{}}}^{2}+u_x+1 \right) ^{2}-2\,{u_{{}}}^{2}}},\\

 {I_2}_x=\dfrac{1}{2},&
{I_2}_u={\dfrac {{u_{{}}}^{2}-u_x+1}{ \left( {u_{{}}}^{2}+u_x-1 \right) ^{2}+2\,{u_{{}}}^{2}}},& {I_2}_{u_x}={\dfrac {u_{{}}}{ \left( {u_{{}}}^{2}+u_x-1
\right) ^{2}+2\,{u_{{}}}^{2}}}.
\end{array}$$
These systems can be solved by quadratures and the solutions provide two independent first integrals for equation (\ref{painleve0}):
$$\begin{array}{l}I_1=-\dfrac{1}{2}\left(x-\sqrt {2}\,{\mbox{arctanh}} \left({\dfrac {{u_{{}}}^{2}+u_x+1}{\sqrt {2}u_{{}}}} \right)\right),\\ I_2=\dfrac{1}{2}\left(x+\sqrt {2}\,{\mbox{arctan}} \left( {\dfrac {{u_{{}}}^{2}+u_x-1}{\sqrt {2}u_{{}}}} \right)\right).
\end{array}$$
\end{enumerate}
Some consequences of the procedure are:
\begin{enumerate}
\item[a)] From $I_1=C_1,I_2=C_2,$ where $C_1,C_2\in \mathbb{R},$ the general solution of (\ref{painleve0}) can be (locally) written in explicit form as follows:
\begin{equation}\label{solgen}
u(x)=\dfrac{\sqrt {2}}{\tanh \left( \frac{\sqrt {2}}{2} \left( x +2C_1\right)\right) -\tan \left( \frac{\sqrt {2}}{2} \left( -x +2C_2\right) \right)},\qquad
C_1,C_2\in\mathbb{R}.
\end{equation}
\item[b)] Two  integrating factors for (\ref{painleve0}) are: 

$$\mu_1=-\frac {u_{{}}}{\left( {u_{{}}}^{2}+u_x+1 \right) ^{2}-2\,{u_{{}}}^{2}}, \qquad \mu_2={\frac {u_{{}}}{\left( {u_{{}}}^{2}+u_x-1 \right)
^{2}+2\,{u_{{}}}^{2}}}.$$
\item[c)] A  Jacobi last multiplier for (\ref{painleve0}) is:
$$M=\frac{-2\,u}{ \left( \left( {u_{{}}}^{2}+u_x-1 \right) ^{2}+2\,{u_{{}}}^{2} \right) \left( \left( {u_{{}}}^{2}+u_x+1 \right) ^{2}-2\,{u_{{} }}^{2} \right) }.$$

\item[d)] Two commuting generalized symmetries for (\ref{painleve0}) are:
$${\mathbf{Z}}_1=\dfrac{1}{2}\left((u^2+u_x-1) ^{2}+u^{2}\right)
\left(\partial_u+\frac {u_x-u^2+1}{u}\partial_{u_x}\right)
$$ 
and
$${\mathbf{Z}}_2=\dfrac{1}{2}\left((u^2+u_x+1) ^{2}-u^{2}\right)
\left(\partial_u+\frac {u_x-u^2-1}{u}\partial_{u_x}\right).
$$

\end{enumerate}

Observe that equation (\ref{painleve0}) admits the Lie point symmetry $\partial_x$ and that this is the unique Lie point symmetry admitted by the equation. The classical Lie method of reduction provides the Abel equation
\begin{equation}\label{reducida}
w_y ={\frac { \left( {y}^
{4}+1 \right) w^{3}}{2 y}}+2\,
w^{2}y-{\frac {w }{2\,y}},
\end{equation} where $y=u$ and $w=1/u_x.$ The general solution of  (\ref{reducida}) can be expressed in the form  $\Delta(y,w(y),C_1)=0,$ where $\Delta$ involves generalized hypergeometric functions. The  general solution of (\ref{painleve0}) arises by solving the corresponding first-order equation $\Delta(u,1/u_x,C_1)=0.$ Due to the involved expression of  $\Delta,$  to perform the quadrature needed for giving such explicit solution does not seem an easy task. Nevertheless, the presented procedure leads to the very compact form (\ref{solgen}) for such general solution.}
\hfill\qedsymbol\end{example}

\begin{example}
\label{painleveotraF}
{\rm  The expressions obtained in Examples \ref{eje1}, \ref{eje3} and \ref{eje4} can be used to integrate any equation of the form (\ref{painleve}),  even if it does not admit Lie point symmetries. This is the case, for instance, of the equation \begin{equation}
\label{painleveparticular}
u_{xx}-\frac{1}{2\,u}u_{x}^2+ 2\,u u_{x} +\frac{1}{2}u^{3}-(2x+1)  u+\frac{1}{2\,u}=0
\end{equation} which corresponds to $F(x)=2x+1.$ 

 For this example, the second-order linear equations (\ref{lineal}) are  the Airy equations
\begin{equation}\label{linealparticular}
\psi_{xx}=\left(1+x\right)\,\psi,\qquad \theta_{xx}=\displaystyle{x}\,\theta.
\end{equation}
Let  $\psi_1(x)={\rm Ai}(1+x),\psi_2(x)={\rm Bi}(1+x)$ and  $\theta_1(x)={\rm Ai}(x),\theta_2(x)={\rm Bi}(x)$ denote linearly independent solutions of the respective Airy equations (\ref{linealparticular}). 

Two independent first integrals of the equation (\ref{painleveparticular}) corresponding to the first integrals (\ref{ipejegeneral}) become:
$$I_1={\frac {-2\,{{\rm Ai}'\left(x+1\right)}u_{{}}+{{\rm Ai}\left(x+1
\right)} \left( {u_{{}}}^{2}+u_x+1 \right) }{-2\,{{\rm Bi}^{(1
)}\left(x+1\right)}u_{{}}+{{\rm Bi}\left(x+1\right)} \left( {u_{{}}}^{
2}+u_x+1 \right) }}
$$
and 
$$
I_2={\frac {-2\,{{\rm Ai}'\left(x\right)}u_{{}}+ \left( {u_{{}}}^{2}+
u_x-1 \right) {{\rm Ai}\left(x\right)}}{-2\,{{\rm Bi}'\left(x
\right)}u_{{}}+ \left( {u_{{}}}^{2}+u_x-1 \right) {{\rm Bi}\left(x
\right)}}}
.
$$
These first integrals provide the general solution of equation (\ref{painleveparticular}) in terms  of Airy functions:
$$u_{{}}(x)=\left({\frac {C_1\,{{\rm Bi}'\left(x+1\right)}-{{\rm Ai}^{(1
)}\left(x+1\right)}}{{C_1}\,{{\rm Bi}\left(x+1\right)}-{{\rm Ai}
\left(x+1\right)}}}-{\frac {{\it C_2}\,{{\rm Bi}'\left(x
\right)}-{{\rm Ai}'\left(x\right)}}{{ C_2}\,{{\rm Bi}\left(x
\right)}-{{\rm Ai}\left(x\right)}}}\right)^{-1},
$$
where $C_1,C_2\in \mathbb{R}.$


Finally, observe that although equation (\ref{painleveparticular}) does not admit Lie point symmetries,  Theorem \ref{teofinal} can be used to construct the following generalized symmetries for equation (\ref{painleveparticular}) 
$${\mathbf{Z}}_1={\frac { \left( -2\,{{\rm Bi}'\left(x+1\right)}u_{{}}+{{\rm Bi}
\left(x+1\right)} \left( {u_{{}}}^{2}+u_x+1 \right)  \right) ^{2}
}{4\,{{\rm Ai}\left(x+1\right)}{{\rm Bi}'\left(x+1\right)}-4\,{
{\rm Bi}\left(x+1\right)}{{\rm Ai}'\left(x+1\right)}}}
\left(\partial_u+\frac {u_x-u^2+1}{u}\partial_{u_x}\right)
$$
and
$${\mathbf{Z}}_2={\frac { \left( -2\,{{\rm Bi}'\left(x\right)}u_{{}}+ \left( {u_{{
}}}^{2}+u_x-1 \right) {{\rm Bi}\left(x\right)} \right) ^{2}}{4\,{
{\rm Bi}'\left(x\right)}{{\rm Ai}\left(x\right)}-4\,{{\rm Ai}'\left(x\right)}{{\rm Bi}\left(x\right)}}}
\left(\partial_u+\frac {u_x-u^2-1}{u}\partial_{u_x}\right).
$$
}\hfill \qedsymbol
\end{example}

\begin{example}
{\rm  In the application of the procedure some difficulties can appear; in this example it is shown how the choosing  of an appropriate   route in the diagram of Figure \ref{diagrama} may help to  overcome these difficulties.
The equation
\begin{equation}
\label{ejemplonuevo}
u_{xx} +{\frac {u_x}{u }}+{\frac {1}{u }}+u=0
\end{equation} does only admit the Lie point symmetry $\mathbf{v}=\partial_x.$ This Lie point symmetry leads to the  reduced equation
\begin{equation}
\label{reducidaejenuevo}
w_y={\frac { \left( {y}^{2}+1
 \right)  w ^{3}}{y}}+{\frac {
w^{2}}{y}},
\end{equation} where $y=u$ and $w=1/u_x.$ Equation (\ref{reducidaejenuevo}) is an Abel equation whose solution can be given in implicit form. By substituting $y=u$ and $w=1/u_x$ into this general solution,  the first-order equation

\begin{equation}\label{auxiliarLie}
{\sqrt{u ^{2}+ ( u_x+1) ^{2}}}-\mbox{arctanh}\left(\dfrac{u_x+1}{\sqrt{u ^{2}+ ( u_x+1) ^{2}}}\right)=C,
\end{equation}
where $C\in \mathbb{R},$ is obtained. If (\ref{auxiliarLie}) is expressed in the form $u_x=H(u,C),$ then the general solution of (\ref{ejemplonuevo}) becomes
$$\int\frac{du}{H(u,C)}=x+K, \quad K\in \mathbb{R}.
$$
The obtained expression for this general solution requires a primitive of the function $1/H.$ Since  $u_x$ cannot be isolated from (\ref{auxiliarLie}),  to obtain the general solution of (\ref{ejemplonuevo}) in a closed form seems an impossible task.

The canonical representative of the equivalence class of the pair $(\partial_x,0)$ is the $\lambda-$symmetry $(\partial_u,\lambda_1),$ for $$\lambda_1=\dfrac{\campoA (Q)}{Q}=-\left(\dfrac{1}{u}+\dfrac {{u_{{}}}^{2}+1}{u_{{x}}u_{{}}}\right),$$ where $Q=-u_x$ is the characteristic of $\partial_x$ and $\campoA$ denotes the vector field associated to (\ref{ejemplonuevo}). This function $\lambda_1$ solves the determining equation (\ref{eqdet2}) for equation (\ref{ejemplonuevo}).
If  another particular solution is searched for  such determining equation, it is not difficult to find a solution which is linear in $u_x:$
$$\lambda_2=\frac{u_x+1}{u}.
$$
Once two $\lambda-$symmetries are known, the algorithm can be started:
\begin{enumerate}
\item  Construct the vector fields (\ref{Ys}):
$$\Yu=\partial_u-\left(\dfrac{1}{u}+{\dfrac {{u_{{}}}^{2}+1}{u_{{x}}u_{{}}}}\right)\partial_{u_x}\quad \mbox{and} \quad \Yd=\partial_u+\dfrac{u_x+1}{u}\partial_{u_x}$$ and the function (\ref{ro}):
$\rho=\dfrac{u_x+1}{u_x\,u}.$
\item  A function $f_1$ satisfying (\ref{f1f2}) can be easily determined: $f_1=u_x.$ However, the determination of a function $f_2$ such that $\Yu(f_2)=\rho f_2$ seems quite complicated. 
The explicit expression of $f_2$ will be skipped, for a moment. 
\item Since $f_1$ is known,  system (\ref{m2}) can be constructed and integrated by qua\-dra\-ture to obtain the solution $w_2=\arctan\left(\dfrac{u}{1+u_x}\right).$

The computation of function $w_1$  cannot be obtained by quadrature from system (\ref{m1}) without the expression of $f_2.$  Nevertheless, observe that the left hand  side of (\ref{auxiliarLie}) defines a function $\widetilde{w}_1$
that is a first integral of $\campoA$ and $\partial_x.$ Since  $(\partial_x,0)\stackrel{\campoA}{\sim}(\partial_u,\lambda_1),$ then $\widetilde{w}_1$ is an invariant for $\Yu.$  In coordinates $\{x,\widetilde{w}_1,w_2\}$ the equation  $\Yu(f_2)=\rho f_2$ is simple and $f_2(x,\widetilde{w}_1,w_2)=\sin(w_2)$ is a particular solution. In the original variables such function becomes
$$f_2(x,u,u_x)=\dfrac{u}{\sqrt{{u ^{2}+ ( u_x+1) ^{2}}}}.$$
Now system (\ref{m2}) is known and can be integrated by quadrature to obtain the solution $$w_1=\sqrt{u ^{2}+ ( u_x+1) ^{2}}-\ln\left|\dfrac{\sqrt{u ^{2}+ ( u_x+1) ^{2}}+u_x+1}{u}\right|.$$

\item  The functions $\phi_1(x,w_2)=\campoA(w_2)=1$ and $\phi_1(x,w_1)=\campoA(w_1)=0$ provide the first-order PDEs
$$ ({{g}_1})_x+({{g}_1})_{w_2}=0  \quad  \mbox{and} \quad
({{g}_2})_x= 0.
$$ The particular solutions  $g_1(x,w_2)=1$ and $g_2(x,w_1)=1$ arise immediately.
\item The functions $f_1,f_2$ of step 2 and 3 and the functions $g_1,g_2$ of step 4  (expressed in terms of $\{x,u,u_x\}$) is all it is  needed to calculate two independent first integrals  $I_1$ and $I_2$ by solving by quadratures  systems (\ref{I1})-(\ref{I2}):
$$\begin{array}{l}
I_1=\sqrt{u ^{2}+ ( u_x+1) ^{2}}-\mbox{arctanh}\left(\dfrac{u_x+1}{\sqrt{u ^{2}+ ( u_x+1) ^{2}}}\right),\\
I_2=x-\arctan\left(\dfrac{u}{u_x+1}\right)
.\end{array}$$
\end{enumerate}
By using these first integrals the general solution of equation (\ref{ejemplonuevo}) can be locally given in explicit form:
 \begin{equation}
 \label{solgenejemplonuevo}
 u(x)=\sin (C_2-x)\Bigl(C_1-\mbox{arctanh}\bigl(\cos(C_2-x)\bigr)\Bigr), \quad C_1,C_2\in\mathbb{R}.
 \end{equation}
 As a remark, expression (\ref{solgenejemplonuevo}) can be used to derived the solution in parametric form of the Abel  equation (\ref{reducidaejenuevo}):
$$\begin{array}{lll}
 y&=&\sin (C_2-x)\Bigl(C_1-\mbox{arctanh}\bigl(\cos(C_2-x)\bigr)\Bigr),\\{}\\
  w&=&\dfrac{1}{\cos (C_2-x)\Bigl(\mbox{arctanh}\bigl(\cos(C_2-x)\bigr)-C_1\Bigr) -1 }.

  \end{array}
  $$
%

}\end{example}

\section{Concluding remarks}

If two non-equivalent $\mathcal{C}^{\infty}-$symmetries for a given second-order ODE  are known, then the independent use of these $\mathcal{C}^{\infty}-$symmetries gives two reduced equations whose integration, not  necessarily by quadrature, can provide the general solution of the ODE. 
In this paper we provide a new method to obtain by quadratures two independent first integrals of the ODE by the combined use of both $\mathcal{C}^{\infty}-$symmetries.

This is done by constructing
a system of two commuting generalized symmetries of the ODE. From these generalized symmetries two independent first integrals of the equation
arise by quadratures.

A comparative study between this new method of integration of the ODE and the reduction methods associated to the given $\mathcal{C}^{\infty}-$symmetries is also provided.
Some relationships between the functions involved in the new method and integrating factors of the reduced and auxiliary equations associated to the
$\mathcal{C}^{\infty}-$symmetries have been established.
Such functions and the $\mathcal{C}^{\infty}-$sy\-mme\-tries provide an explicit formula for a Jacobi last multiplier that generalizes (for $n=2$) the classical expression in terms of Lie
point symmetries.

The results have been illustrated with a family of equations of the XXVII case in the  Painlev\'e-Gambier classification. By using two $\lambda-$symmetries and two fundamental sets of solutions of two second-order linear ODEs, the equations can be integrated by quadratures. Explicit expressions for the first integrals, integrating factors, a Jacobi last multiplier and the general solution have also been provided as by-products of the quadrature process.

The method has been successfully applied to integrate by quadratures $2$nd-order ODEs lacking Lie point symmetries or admitting just one Lie point symmetry.

The results presented in this paper could provide novel methods to find exact solutions of nonlinear
ordinary and partial differential equations as well as establish new connections between analytical methods which are widely used in the contemporary literature \cite{mohanasubha2014interplay,mohanasubha2015certain,bhuvaneswari2012application}.

\section*{Acknowledgements} The authors acknowledge the financial support from the University of Cádiz and from Junta de Andalucía to the research group FQM 377.  A. Ruiz  acknowledges the support of a grant of the University of C\'adiz program ``Movilidad Internacional, Becas UCA- Internacional-Posgrado'' during his stay at the  University of Minnesota.
\bibliographystyle{elsarticle-num-names}

\begin{thebibliography}{37}
\expandafter\ifx\csname natexlab\endcsname\relax\def\natexlab#1{#1}\fi
\providecommand{\url}[1]{\texttt{#1}}
\providecommand{\href}[2]{#2}
\providecommand{\path}[1]{#1}
\providecommand{\DOIprefix}{doi:}
\providecommand{\ArXivprefix}{arXiv:}
\providecommand{\URLprefix}{URL: }
\providecommand{\Pubmedprefix}{pmid:}
\providecommand{\doi}[1]{\href{http://dx.doi.org/#1}{\path{#1}}}
\providecommand{\Pubmed}[1]{\href{pmid:#1}{\path{#1}}}
\providecommand{\bibinfo}[2]{#2}
\ifx\xfnm\relax \def\xfnm[#1]{\unskip,\space#1}\fi
\bibitem[{Olver(1993)}]{olver1993applications}
\bibinfo{author}{P.~Olver}, \bibinfo{title}{Applications of {L}ie groups to
  differential equations}, \bibinfo{edition}{2nd} ed.,
  \bibinfo{publisher}{Springer-Verlag}, \bibinfo{address}{New York},
  \bibinfo{year}{1993}.
\bibitem[{Ibragimov(2010)}]{ibragimovlibro}
\bibinfo{author}{N.~H. Ibragimov}, \bibinfo{title}{A practical course in
  differential equations and mathematical modelling}, \bibinfo{publisher}{World Scientific},
  \bibinfo{address}{Beijing}, \bibinfo{year}{2010}.
\bibitem[{Stephani(1989)}]{stephani}
\bibinfo{author}{H. Stephani},
\bibinfo{title}{{D}ifferential equations: their solution using symmetries},
\bibinfo{publisher}{Cambridge University Press}, \bibinfo{address}{New York}, \bibinfo{year}{1989}.
\bibitem[{Gonz{\'a}lez-L{\'o}pez(1988)}]{artemioecu}
\bibinfo{author}{A.~Gonz{\'a}lez-L{\'o}pez},
\newblock \bibinfo{title}{Symmetry and integrability by quadratures of ordinary
  differential equations},
\newblock \bibinfo{journal}{Phys. Lett. A} \bibinfo{volume}{133}
  (\bibinfo{year}{1988}) \bibinfo{pages}{190--194}.
  \DOIprefix\doi{10.1016/0375-9601(88)91015-8}.
\bibitem[{Gonz{\'a}lez-Gasc{\'o}n and
  Gonz{\'a}lez-L{\'o}pez(1988)}]{artemiosis}
\bibinfo{author}{F.~Gonz{\'a}lez-Gasc{\'o}n},
  \bibinfo{author}{A.~Gonz{\'a}lez-L{\'o}pez},
\newblock \bibinfo{title}{Newtonian systems of differential equations,
  integrable via quadratures, with trivial group of point symmetries},
\newblock \bibinfo{journal}{Phys. Lett. A} \bibinfo{volume}{129}
  (\bibinfo{year}{1988}) \bibinfo{pages}{153--156}.
  \DOIprefix\doi{10.1016/0375-9601(88)90134-X}.
\bibitem[{Govinder and Leach(1997)}]{govinder97}
\bibinfo{author}{K.~S. Govinder}, \bibinfo{author}{P.~G.~L. Leach},
\newblock \bibinfo{title}{A group-theoretic approach to a class of second-order
  ordinary differential equations not possessing {L}ie point symmetries},
\newblock \bibinfo{journal}{J. Phys. A: Math. Gen.} \bibinfo{volume}{30}
  (\bibinfo{year}{1997}) \bibinfo{pages}{2055--2068}.
  \DOIprefix\doi{10.1088/0305-4470/30/6/026}.
\bibitem[{Nucci(2008)}]{nucci2008}
\bibinfo{author}{M.~C. Nucci},
\newblock \bibinfo{title}{{L}ie symmetries of a {P}ainlev\'e-type equation
  without {L}ie symmetries},
\newblock \bibinfo{journal}{J. Nonlinear Math. Phys.} \bibinfo{volume}{15 (2)}
  (\bibinfo{year}{2008}) \bibinfo{pages}{201--211}.
    \DOIprefix\doi{10.2991/jnmp.2008.15.2.7}.
\bibitem[{Muriel and Romero(2001)}]{muriel01ima1}
\bibinfo{author}{C.~Muriel}, \bibinfo{author}{J.~L. Romero},
\newblock \bibinfo{title}{New methods of reduction for ordinary differential
  equations},
\newblock \bibinfo{journal}{IMA J. Appl. Math.} \bibinfo{volume}{66}
  (\bibinfo{year}{2001}) \bibinfo{pages}{111--125}.
  \DOIprefix\doi{10.1093/imamat/66.2.111}.
\bibitem[{Muriel and Romero(2003)}]{muriel03lie}
\bibinfo{author}{C.~Muriel}, \bibinfo{author}{J.~L. Romero},
\newblock \bibinfo{title}{{$C^\infty$}-{S}ymmetries and reduction of equations
  without {L}ie point symmetries},
\newblock \bibinfo{journal}{J. {L}ie Theory} \bibinfo{volume}{13}
  (\bibinfo{year}{2003}) \bibinfo{pages}{167--188}.
\bibitem[{{Gaeta}(2009)}]{gaetatwisted}
\bibinfo{author}{G.~{Gaeta}},
\newblock \bibinfo{title}{{Twisted symmetries of differential equations}},
\newblock \bibinfo{journal}{J. Nonlinear Math. Phys.} \bibinfo{volume}{16}
  (\bibinfo{year}{2009}) \bibinfo{pages}{107--136}.
  \DOIprefix\doi{10.1142/S1402925109000352}.
\bibitem[{Ferraioli and Morando (2009)}]{1751-8121-42-3-035210}
\bibinfo{author}{D.~C. Ferraioli}, \bibinfo{author}{P.~Morando},
  \bibinfo{title}{Local and nonlocal solvable structures in the
  reduction of {ODE}s},
  \bibinfo{journal}{J. Phys. A: Math. Theor.} \bibinfo{volume}{42}
  (\bibinfo{year}{2009}{\natexlab{a}}) \bibinfo{pages}{035210}.    \DOIprefix\doi{10.1088/1751-8113/42/3/035210}.
\bibitem[{Ferraioli and Morando (2009)}]{doi:10.1142/S1402925109000303}
\bibinfo{author}{D.~C. Ferraioli}, \bibinfo{author}{P.~Morando},
  \bibinfo{title}{Applications of solvable structures to the nonlocal
  symmetry-reduction of {ODE}s},
  \bibinfo{journal}{J. Nonlinear Math. Phys.} \bibinfo{volume}{16}
  (\bibinfo{year}{2009}{\natexlab{b}}) \bibinfo{pages}{27--42}.
  \DOIprefix\doi{10.1142/S1402925109000303}.
\bibitem[{Gaeta and Morando (2004)}]{0305-4470-37-27-007}
\bibinfo{author}{G. Gaeta, P. Morando},
  \bibinfo{title}{On the geometry of lambda-symmetries and PDE reduction},
  \bibinfo{journal}{J. Phys. A: Math. Gener.} \bibinfo{volume}{37} (27)
  (\bibinfo{year}{2004}) \bibinfo{pages}{6955}.
    \DOIprefix\doi{10.1088/0305-4470/37/27/007}.
    \bibitem[{Muriel and Romero(2009)}]{muriel2009}
\bibinfo{author}{C.~Muriel}, \bibinfo{author}{J.~L. Romero},
\newblock \bibinfo{title}{First integrals, integrating factors and
  {$\lambda$}-symmetries of second-order differential equations},
\newblock \bibinfo{journal}{J. Phys. A: Math. Theor.} \bibinfo{volume}{42}
  (\bibinfo{year}{2009}) \bibinfo{pages}{365207, 17}.
  \DOIprefix\doi{10.1088/1751-8113/42/36/365207}.
\bibitem[{Whittaker and McCrae(1988)}]{whittaker1988treatise}
\bibinfo{author}{E.~Whittaker}, \bibinfo{author}{W.~McCrae}, \bibinfo{title}{A
  Treatise on the Analytical Dynamics of Particles and Rigid Bodies}, Cambridge
  Mathematical Library, \bibinfo{publisher}{Cambridge University Press},
  \bibinfo{year}{1988}.
\bibitem[{Nucci(2005)}]{nucci2005jacobi}
\bibinfo{author}{M.~Nucci},
\newblock \bibinfo{title}{{J}acobi last multiplier and {L}ie symmetries: a
  novel application of an old relationship},
\newblock \bibinfo{journal}{J. Nonlinear Math. Phys.} \bibinfo{volume}{12}
  (\bibinfo{year}{2005}) \bibinfo{pages}{284--304}.
  \DOIprefix\doi{10.2991/jnmp.2005.12.2.9}.
\bibitem[{Nucci and Leach(2009)}]{MR2606129}
\bibinfo{author}{M.~C. Nucci}, \bibinfo{author}{P.~G.~L. Leach},
\newblock \bibinfo{title}{An old method of {J}acobi to find {L}agrangians},
\newblock \bibinfo{journal}{J. Nonlinear Math. Phys.} \bibinfo{volume}{16}
  (\bibinfo{year}{2009}) \bibinfo{pages}{431--441}.
  \DOIprefix\doi{10.1142/S1402925109000467}.
\bibitem[{Nucci and Leach(2002)}]{NucciLeachCGroup1}
\bibinfo{author}{M.~C. Nucci}, \bibinfo{author}{P.~G.~L. Leach},
\newblock \bibinfo{title}{{J}acobi's last multiplier and the complete symmetry
  group of the {E}uler-{P}oinsot system},
\newblock \bibinfo{journal}{J. Nonlinear Math. Phys.} \bibinfo{volume}{9}
  (\bibinfo{year}{2002}) \bibinfo{pages}{110--121}.
  \DOIprefix\doi{10.2991/jnmp.2002.9.s2.10}.
\bibitem[{Nucci and Leach(2005)}]{NucciLeachCGroup2}
\bibinfo{author}{M.~C. Nucci}, \bibinfo{author}{P.~G.~L. Leach},
\newblock \bibinfo{title}{{J}acobi's last multiplier and the complete symmetry
  group of the {E}rmakov-{P}inney equation},
\newblock \bibinfo{journal}{J. Nonlinear Math. Phys.} \bibinfo{volume}{12}
  (\bibinfo{year}{2005}) \bibinfo{pages}{305--320}.
  \DOIprefix\doi{10.2991/jnmp.2005.12.2.10}.
\bibitem[{Ince(1956)}]{ince}
\bibinfo{author}{E.~Ince}, \bibinfo{title}{Ordinary Differential Equations},
  Dover Books on Science, \bibinfo{publisher}{Dover Publications},
  \bibinfo{address}{New York}, \bibinfo{year}{1956}.
\bibitem[{Guha et~al.(2013)Guha, Choudhury, and Khanra}]{guha2013lambda}
\bibinfo{author}{P.~Guha}, \bibinfo{author}{A.~Choudhury},
  \bibinfo{author}{B.~Khanra},
\newblock \bibinfo{title}{$\lambda-${S}ymmetries, isochronicity and integrating
  factors of nonlinear ordinary differential equations},
\newblock \bibinfo{journal}{J. Eng. Math.} \bibinfo{volume}{82}
  (\bibinfo{year}{2013}) \bibinfo{pages}{85--99}.
  \DOIprefix\doi{10.1007/s10665-012-9614-5}.
\bibitem[{Muriel and Romero(2010)}]{muriel09wascom}
\bibinfo{author}{C.~Muriel}, \bibinfo{author}{J.~L. Romero},
\newblock \bibinfo{title}{{$\lambda$}-{S}ymmetries on the derivation of first
  integrals of ordinary differential equations},
  \bibinfo{publisher}{World Sci. Publ., Hackensack, NJ}, \bibinfo{year}{2010},
  pp. \bibinfo{pages}{303--308}. \DOIprefix\doi{10.1142/9789814317429-0041}.
\bibitem[{Muriel and Romero(2008)}]{muriel08}
\bibinfo{author}{C.~Muriel}, \bibinfo{author}{J.~L. Romero},
\newblock \bibinfo{title}{Integrating factors and {$\lambda$}-symmetries},
\newblock \bibinfo{journal}{J. Nonlinear Math. Phys.} \bibinfo{volume}{15}
  (\bibinfo{year}{2008}) \bibinfo{pages}{300--309}.
  \DOIprefix\doi{10.2991/jnmp.2008.15.s3.29}.
\bibitem[{Bluman and Anco(2002)}]{blumanlibro}
\bibinfo{author}{G.~W. Bluman}, \bibinfo{author}{S.~C. Anco},
  \bibinfo{title}{Symmetry and Integration Methods for Differential Equations},
  \bibinfo{publisher}{Springer-Verlag}, \bibinfo{address}{New York},
  \bibinfo{year}{2002}.
\bibitem[{Anco and Bluman(1998)}]{blumanancoarticulo}
\bibinfo{author}{S.~C. Anco}, \bibinfo{author}{G.~Bluman},
\newblock \bibinfo{title}{Integrating factors and first integrals for ordinary
  differential equations},
\newblock \bibinfo{journal}{Euro. J. Appl. Math.} \bibinfo{volume}{9}
  (\bibinfo{year}{1998}) \bibinfo{pages}{245--259}.
\bibitem[{Muriel and Romero(2014)}]{muriel2014Jacobi}
\bibinfo{author}{C.~Muriel}, \bibinfo{author}{J.~Romero},
\newblock \bibinfo{title}{The $\lambda$-symmetry reduction method and {J}acobi
  last multipliers},
\newblock \bibinfo{journal}{Commun. Nonlinear Sci. Numer. Simul.}
  \bibinfo{volume}{19} (\bibinfo{year}{2014}) \bibinfo{pages}{807--820}.
  \DOIprefix\doi{10.1016/j.cnsns.2013.07.027}.
  
  
  \bibitem{mohanasubha2014interplay}
  \bibinfo{author}{R. Mohanasubha, V. K. Chandrasekar, M. Senthilvelan, M, Lakshmanan},
    \bibinfo{title}{Interplay of symmetries, null forms, Darboux polynomials, integrating factors and Jacobi multipliers in integrable second-order differential equations},
    \bibinfo{journal}{ P. Roy. Soc. A}
    \bibinfo{volume}{470} (\bibinfo{year}{2014}) \bibinfo{pages}{20130656}.
    \DOIprefix\doi{10.1098/rspa.2013.0656}.
    
  \bibitem{mohanasubha2015certain}
  \bibinfo{author}{R. Mohanasubha, V. K. Chandrasekar, M. Senthilvelan, M, Lakshmanan},
    \bibinfo{title}{On certain analytical methods in finding integrable systems and their interconnections},
    \bibinfo{journal}{arXiv:1502.03914}  (\bibinfo{year}{2015}).
    
  
    \bibitem{bhuvaneswari2012application}
    \bibinfo{author}{A. Bhuvaneswari, R. A. Kraenkel, M. Senthilvelan},
      \bibinfo{title}{Application of the $\lambda$-symmetries approach and time independent integral of the modified Emden equation},
      \bibinfo{journal}{Nonlinear Analysis: Real World Applications}
      \bibinfo{volume}{3} (\bibinfo{year}{2012}) \bibinfo{pages}{1102--1114}.
      \DOIprefix\doi{10.1016/j.nonrwa.2011.08.030}.
      
  

\end{thebibliography}
\def\cprime{$'$}

\end{document}